\documentclass[10pt]{article}
\usepackage{graphicx}
\usepackage{amsmath}
\usepackage{amsfonts}
\usepackage{amssymb}
\usepackage{amsthm} 
\usepackage[shortlabels]{enumitem}
\usepackage{multicol}
\usepackage[cp1250]{inputenc}
\usepackage[a4paper, left=2.5cm, right=2.5cm, top=3.5cm, bottom=3.5cm, headsep=1.2cm]{geometry}

\usepackage{caption}
\usepackage{color}
\definecolor{sepia}{cmyk}{0, 0.83, 1, 0.70}

\usepackage[plainpages=false,pdfpagelabels,bookmarksnumbered,%
 colorlinks=true,%
 linkcolor=sepia,%
 citecolor=sepia,%
 unicode]{hyperref}

\newtheorem{theorem}{Theorem}

\newtheorem{lemma}[theorem]{Lemma}

\newtheorem{proposition}[theorem]{Proposition}

\theoremstyle{definition}

\newtheorem{example}[theorem]{Example}

\newtheorem{remark}[theorem]{Remark}

\numberwithin{equation}{section} 
\numberwithin{theorem}{section}  

\makeatletter
\renewenvironment{proof}[1][\proofname]
{\par
	\pushQED{$\blacksquare$} 
	\normalfont\topsep6\p@\@plus6\p@\relax
	\trivlist
	\item[\hskip\labelsep\bfseries#1\@addpunct{.}]
	\ignorespaces}
{\popQED \endtrivlist\@endpefalse}
\makeatother

\DeclareMathOperator*{\fix}{Fix}

\definecolor{myblue}{rgb}{.8, .8, 1}

\begin{document}

\title{\vspace{-5em} \textbf{Error Bounds for the Method of Simultaneous Projections with Infinitely Many Subspaces}}

\author{Simeon Reich\thanks{Department of Mathematics,
    The Technion -- Israel Institute of Technology, 3200003 Haifa, Israel;
    E-mail: \texttt{sreich@technion.ac.il}. } \and
    Rafa\l\ Zalas\thanks{Department of Mathematics,
    The Technion -- Israel Institute of Technology, 3200003 Haifa, Israel;
    E-mail: \texttt{zalasrafal@gmail.com}. }
}

\maketitle

\begin{abstract}
  We investigate the properties of the simultaneous projection method as applied to countably infinitely many closed and linear subspaces of a real Hilbert space. We establish the optimal error bound for linear convergence of this method, which we express in terms of the  cosine of the Friedrichs angle  computed in an infinite product space. In addition, we provide  estimates and alternative expressions for the above-mentioned number.  Furthermore, we relate this number to the dichotomy theorem and to super-polynomially fast convergence. We also discuss polynomial convergence of the simultaneous projection method which takes place for particularly chosen starting points.

  \bigskip
  \noindent \textbf{Key words and phrases:} Friedrichs angle; Product space; Rates of convergence; Simultaneous projection method

  \bigskip
  \noindent \textbf{2010 Mathematics Subject Classification:} 41A25, 41A28, 41A44, 41A65.
\end{abstract}

\section{Introduction}

Let $\mathcal H$ be a real Hilbert space with its inner product denoted by $\langle \cdot, \cdot \rangle$ and the induced norm denoted by $\|\cdot\|$. In this paper we study the asymptotic properties of the simultaneous projection method as applied to a possibly  countably infinite number of closed and linear subspaces of $\mathcal H$. We begin by briefly recalling some of the known results which have so far been established only for a finite number of subspaces. Moreover, we recall relevant results related to the cyclic projection method. We do not discuss here the case of general closed and convex sets, for which we refer the  interested reader to \cite{BauschkeBorwein1996, BauschkeNollPhan2015, BorweinLiTam2017, Cegielski2012, CegielskiReichZalas2018}.  For other examples of projection methods studied in the setting of closed and linear subspaces, we refer the reader to \cite{ArtachoCampoy2019, BadeaSeifert2017, BauschkeCruzNghiaPhanWang2014, BauschkeCruzNghiaPhanWang2016, Tam2020}.

\subsection{Related Work}

For now let $r \in \mathbb Z_+$.  For each $i =1,2, \ldots, r$, let $M_i\subset \mathcal H$ be a nontrivial closed and linear subspace, and let $P_{M_i}$ denote the corresponding orthogonal projection. Moreover, let $M:=\bigcap_{i=1}^r M_i$ with the corresponding orthogonal projection $P_M$. In the next three theorems, the operator $T_r$ can be either the cyclic projection operator $T_r:=P_{M_r}\ldots P_{M_1}$ or the simultaneous projection operator $T_r:=\frac 1 r \sum_{i=1}^r P_{M_i}$. In particular, $T_2=P_{M_2}P_{M_1}$ is the alternating projection operator. We begin with a well-known result.

\begin{theorem}[Norm Convergence]\label{int:th:norm}
For each $x\in\mathcal H$, we have
\begin{equation}
\lim_{k\rightarrow\infty}\left\|T_r^k (x)-P_M(x)\right\| = 0.
\end{equation}
\end{theorem}
Theorem \ref{int:th:norm} goes back to von Neumann \cite{Neumann1949}, when $T_2=P_{M_2}P_{M_1}$, and to Halperin \cite{Halperin1962}, when $T_r=P_{M_r}\ldots P_{M_1}$. Lapidus \cite{Lapidus1981} and Reich \cite{Reich1983} proved Theorem \ref{int:th:norm} for $T_r=\frac 1 r \sum_{i=1}^r P_{M_i}$.

It turns out that in the infinite dimensional case, the convergence properties can indeed differ from their finite dimensional counterparts.

\begin{theorem}[Dichotomy]\label{int:th:dichotomy}
Exactly one of the following two statements holds:
\begin{enumerate}[(i)]
  \item $\sum_{i=1}^r M_i^\perp$ is closed. Then  the sequence  $\{T_r^k\}_{k=1}^\infty$  converges linearly to $P_M$.
  \item $\sum_{i=1}^r M_i^\perp$ is not closed. Then  the sequence  $\{T_r^k\}_{k=1}^\infty$   converges arbitrarily slowly to $P_M$.
\end{enumerate}
\end{theorem}

We recall that the linear convergence in (i) means that there are constants $c>0$ and $q\in (0,1)$ such that the inequality $\|T_r^k(x)-P_M(x)\|\leq c q^k \|x\|$ holds for all $k=1,2,\ldots$  and all $x \in \mathcal H$.  The arbitrarily slow convergence in (ii) means that for any sequence of scalars $\{a_k\}_{k=1}^\infty$ with $0\leq a_k$ and $a_k \to 0$, there is a point $x \in \mathcal H$ such that the inequality
$\|T_r(x) - P_{M}(x)\| \geq a_k$ holds for all $k=1,2,\ldots$.

The first instance of  Theorem \ref{int:th:dichotomy} (ii)  is due to Bauschke, Deutsch and Hundal \cite{BauschkeDeutschHundal2009}, who proved it for the alternating projection method ($T_2=P_{M_2} P_{M_1}$) with decreasing null sequences $\{a_k\}_{k=1}^\infty$. These authors commented that their result is also valid for $T_r=\frac 1 r \sum_{i=1}^r P_{M_i}$ with $r\geq 2$ because of the connection between the method of simultaneous projections and the method of alternating projections in the product space. We return to this connection below. The statement of Theorem \ref{int:th:dichotomy} with $T_r=P_{M_r}\ldots P_{M_1}$, allowing $r\geq 2$ and any null nonnegative sequence, has been established by Deutsch and Hundal in \cite{DeutschHundal2010}. Similar results can be found, for example, in \cite{BadeaGrivauxMuller2011, BadeaSeifert2016, BadeaSeifert2017, DeutschHundal2011, DeutschHundal2015}.

 Despite the arbitrarily slow convergence presented in alternative (ii), there do exist sets of starting points in $\mathcal H$ for which there are relatively good error upper bounds. We comment on this matter in Theorems \ref{int:th:superPoly} and \ref{int:th:poly}.

\begin{theorem}[Super-polynomial Rate] \label{int:th:superPoly}
If $\sum_{i=1}^r M_i^\perp$ is not closed, then  the sequence $\{T_r^k\}_{k=1}^\infty$  converges super-polynomially fast to $P_M$ on some dense linear subspace $X\subset\mathcal H$.
\end{theorem}

The super-polynomially fast convergence means that $k^n \|T_r^k(x)-P_M(x)\|\to 0$ as $k\to\infty$ for each $x\in X$ and for all $n=1,2,\ldots$. Theorem \ref{int:th:superPoly} is due to Badea and Seifert \cite{BadeaSeifert2016}, who established it for $T_r:=P_{M_r}\ldots P_{M_1}$ in a complex Hilbert space. By using a complexification argument, we see that this result is also valid in a real Hilbert space. Similarly to the case of Theorem \ref{int:th:dichotomy}, the result holds for $T_r=\frac 1 r \sum_{i=1}^r P_{M_i}$ as can be seen by using the product space approach. The details can be found in \cite{ReichZalas2017}.

 The following theorem has recently been established by Borodin and Kopeck\'{a} in \cite{BorodinKopecka2020}.

\begin{theorem}[Polynomial Rate]\label{int:th:poly}
Assume that $M_1 \cap M_2 = \{0\}$.  Then for any  $x \in M_1^\perp + M_2^\perp$  there is $C(x)>0$ such that
\begin{equation}\label{int:th:poly:ineq}
  \|(P_{M_2}P_{M_1})^k(x)\| \leq \frac{C(x)}{\sqrt k}, \quad k=1,2,\ldots.
\end{equation}
Moreover,  when $\mathcal H$ is infinite dimensional,  the denominator $\sqrt k$ cannot be replaced by $k^{1/2+\varepsilon}$ for any $\varepsilon>0$  (that is, for each $\varepsilon >0$ there are two closed linear subspaces $M_1$, $M_2$ and $x \in M_1^\perp + M_2^\perp$ such that $\|(P_{M_2}P_{M_1})^k(x)\| \geq C(x) k^{-(1/2+\varepsilon)}$ for some $C(x) > 0$ and all $k = 1,2,\ldots$).
\end{theorem}
It is not difficult to see that estimate \eqref{int:th:poly:ineq} also holds when  $M_1^\perp \cap M_2^\perp\neq \{0\}$.  We comment on this in the proof of Theorem \ref{th:polyRate} below.

We now return to the case where $\sum_{i=1}^{r}M_i^\perp$ is closed. In this case one may be interested in finding the optimal error bound, that is, the smallest possible estimate for the relative error $e_k(x):=\|T_r^k(x)-P_M(x)\|/\|x\|$, which is independent of $x$. The answer to this question leads to the computation of the operator norm since $\sup_{x\neq 0} e_k(x) = \|  T_r^k  -P_M\|$.

The first result of this type for the alternating projections method (APM) is due to Aronszajn (inequality) \cite{Aronszajn1950}, and Kayalar and Weinert (equality) \cite{KayalarWeinert1988}, who expressed the optimal error bound in terms of the cosine of the Friedrichs angle between the subspaces $M_1$ and $M_2$, which we denote by $\cos(M_1,M_2)$. Recall that
\begin{equation}\label{int:def:cosM1M2}
  \cos(M_1,M_2)
  := \sup \left\{ \langle x,  y \rangle \colon
  \begin{array}{l}
    x \in M_1 \cap (M_1 \cap M_2)^\perp \cap B,\\
    y \in M_2 \cap (M_1 \cap M_2)^\perp \cap B
  \end{array}
  \right\} \in [0,1],
\end{equation}
where $B:=\{x \in \mathcal H \colon \|x\|\leq 1\}$. Their result reads as follows:

\begin{theorem}[Optimal Error Bound] \label{int:th:APM}
For each $k=1,2,\ldots,$ we have
\begin{equation}\label{int:th:APM:eq}
\|(P_{M_2}P_{M_1})^k-P_M\|=\cos(M_1,M_2)^{2k-1}.
\end{equation}
\end{theorem}
Only estimates are known for $r>2$; see, for example, \cite{KayalarWeinert1988, PustylnikReichZaslavski2012} for those which involve angles measured between $M_1\cap\ldots\cap M_i$ and $M_{i+1}$, $i=1,\ldots,r-1$, and \cite{BadeaGrivauxMuller2011, PustylnikReichZaslavski2013} for those which are expressed using the so-called inclination number.  At this point, recall that
\begin{equation}\label{int:equivalence}
  \cos(M_1,M_2)<1 \quad \Longleftrightarrow \quad M_1^\perp+M_2^\perp \text{ is closed};
\end{equation}
see, for example, \cite{Deutsch1985, BauschkeBorwein1993} and \cite[Theorem 9.35 and p. 235]{Deutsch2001} for detailed historical notes. Equivalence \eqref{int:equivalence} may also be deduced from Theorems \ref{int:th:dichotomy} and \ref{int:th:APM}.

A result analogous to Theorem \ref{int:th:APM} has been established in \cite{ReichZalas2017} for the simultaneous projection method. Indeed, let the product space $\mathbf H_r := \mathcal H^r = \mathcal H \times \ldots \times \mathcal H$ be equipped with the inner product $\langle \mathbf x, \mathbf y \rangle_r :=\frac 1 r \sum_{i=1}^{r} \langle x_i,y_i\rangle$ and the induced norm $\|\mathbf x\|_r:=\sqrt{\langle \mathbf x, \mathbf x\rangle_r}$, where $\mathbf x = \{x_1,\ldots,x_r\}$, $\mathbf y=\{y_1,\ldots,y_r\}$. Moreover, let $\mathbf C_r:=M_1\times\ldots\times M_r \subset \mathbf H_r$ and $\mathbf D_r:=\{\{x,\ldots,x\}\colon x\in \mathcal H\} \subset \mathbf H_r$, and denote by $\cos_r(\mathbf C_r, \mathbf D_r)$ the corresponding cosine of the Friedrichs angle in $\mathbf H_r$; see \eqref{def:cosM1N2inHr} and Remark \ref{rem:notation}.

\begin{theorem}[Optimal Error Bound] \label{int:th:SPM}
For each $k=1,2,\ldots,$ we have
\begin{equation}\label{int:th:SPM:eq}
\left\|\left(\frac 1 r \sum_{i=1}^r P_{M_i}\right)^k - P_{M}\right\|=\cos_r(\mathbf C_r, \mathbf D_r)^{2k}.
\end{equation}
In particular, when $r = 2$, we get $\cos_2(\mathbf C_2, \mathbf D_2)^2  = \frac 1 2 + \frac 1 2 \cos(M_1,M_2)$.
\end{theorem}

Recall that the alternating projection formalization introduced above is due to Pierra \cite{Pierra1984}, who observed that for each $x\in \mathcal H$, $\mathbf x = (x,\ldots,x)$ and $k=1,2,\ldots$, we have
$\|( \frac 1 r \sum_{i=1}^r P_{M_i})^k(x) - P_{M}(x)\| = \|(P_{\mathbf D_r} P_{\mathbf C_r})^k(\mathbf x) - P_{\mathbf C_r \cap \mathbf D_r}(\mathbf x)\|_r.$ Note here that by simply combining this with Theorem \ref{int:th:APM}, we only obtain an upper bound given by $\cos_r(\mathbf C_r, \mathbf D_r)^{2k-1}$.  The properties of the cosine $\cos_r(\mathbf C_r, \mathbf D_r)$ were studied in \cite{BadeaGrivauxMuller2011}, where equality \eqref{int:th:SPM:eq} was shown for $k = 1$.

It turns out that when $r = 2$, the cosine of the Friedrichs angle appears in the optimal rate estimates for many other well-known projection methods. See, for example, \cite{BauschkeCruzNghiaPhanWang2016} for the relaxed alternating projection method, \cite{BauschkeCruzNghiaPhanWang2014} for the Douglas-Rachford method or \cite{ArtachoCampoy2019} for the method of averaged alternating modified reflections. We refer the interested reader to \cite[Table 1]{ArtachoCampoy2019}, where one can find an elegant comparison of rates.

\subsection{Contribution and Organization of the Paper}

The purpose of the present paper is to investigate the asymptotic properties of the simultaneous projection operator, analogous to those mentioned above, in the case where the number of subspaces $M_i$ is possibly countably infinite, that is, when $r \in \mathbb Z_+\cup\{\infty\}$. The aforesaid operator is defined by $T_\omega:=\sum_{i=1}^r \omega_i P_{M_i}$, where $\omega = \{\omega_i\}_{i=1}^r$ is a vector/sequence of weights $\omega_i \in (0,1)$ the sum of which equals one.

We carry out our study by adjusting the product space formalization of Pierra. For this purpose, for each operator $T_\omega$, we define the weighted product space $(\mathbf H_\omega, \langle \cdot, \cdot \rangle_\omega)$, which is the analogue of $(\mathbf H_r, \langle \cdot, \cdot \rangle_r)$, the subspaces $\mathbf C_\omega$ and $\mathbf D_\omega$ in $\mathbf H_\omega$, which correspond to $\mathbf C_r$ and $\mathbf D_r$ in $\mathbf H_r$ and finally, the cosine of the Friedrichs angle $\cos_\omega(\mathbf C_\omega, \mathbf D_\omega)$, which is an analogue of $\cos_r(\mathbf C_r, \mathbf D_r)$; see Section \ref{sec:preliminaries} for fully-fledged definitions, notation and basic properties. We note here only that for $r = \infty$ the product space $\mathbf H_\omega$ coincides with $\ell^2_\omega(\mathcal H):=\{\mathbf x = \{x_i\}_{i=1}^\infty \colon  x_i\in\mathcal H,\  i=1,2,\ldots,\  \sum_{i=1}^{\infty}\omega_i\|x_i\|^2 <\infty\}$.

We begin this study in Section \ref{sec:Pierra} by showing the explicit connection between the operator $T_\omega$, the projection onto $M$ and the projections onto $\mathbf C_\omega$, $\mathbf D_\omega$ and $\mathbf C_\omega \cap \mathbf D_\omega$. Within this framework, we establish that the iterates of the simultaneous projection method $\{T_\omega^k(x)\}_{k=1}^\infty$ converge in norm to $P_M(x)$ for each starting point $x \in \mathcal H$, even when $r=\infty$. Moreover, using the powers of $\cos_\omega(\mathbf C_\omega, \mathbf D_\omega)^2$, we find an expression for the norm $\|T_\omega^k - P_M\|$, which, when smaller than $1$, becomes the optimal error bound for linear convergence.

In Section \ref{sec:cosCD} we present a detailed study of the cosine $\cos_\omega(\mathbf C_\omega, \mathbf D_\omega)$. In particular, we provide an alternative formula for it, where the supremum is taken over a possibly smaller set (Lemma \ref{lem:cosCD1}). Furthermore, we find a new estimate, which, depending on the weights $\omega$, may hold as a strict inequality, or as an equality (see Lemma \ref{lem:cosCD2} and Example \ref{ex:orthogonal}). The important property of this estimate is that it must hold as an equality whenever the subspaces $\mathbf C_\omega$ and $\mathbf D_\omega$ are parallel and in this case the equality holds for all weights $\omega$ (Theorem \ref{thm:cosCDequal1}). On the other hand, we show that the cosine can be easily evaluated when the subspaces $M_i$ are pairwise orthogonal (Proposition \ref{prop:orthogonal}). In addition, we point out that by reducing the multiple copies of the subspaces $M_i$, the cosine $\cos_\omega(\mathbf C_\omega, \mathbf D_\omega)$ can be computed in a simpler manner. For this reason we introduce a rearrangement lemma (see the \hyperref[sec:Appendix]{Appendix}). On the other hand, when $r = \infty$, we can approximate the cosine $\cos_\omega(\mathbf C_\omega, \mathbf D_\omega)$ by a limit process of cosines between $\mathbf C_q$ and $\mathbf D_q$ defined in smaller product spaces, where $q\in\mathbb Z_+$ and $q\to \infty$.

In Section \ref{sec:AsymptoticProp} we return to the asymptotic properties of the simultaneous projection method. Building on the idea of $\ell^2$-summability, we replace the subspace $\sum_{i=1}^{r}M_i^\perp$, which plays a central role in Theorems \ref{int:th:dichotomy}--\ref{int:th:poly}, by another $\omega$-dependent subspace, which for $r = \infty$ becomes $\{\sum_{i=1}^{\infty} \omega_i x_i \colon x_i \in M_i^\perp,\ i=1,2,\ldots,\ \sum_{i=1}^{\infty}\omega_i\|x_i\|^2<\infty\}$; see \eqref{prop:Minkowski:Aw}. We show that this subspace, which we denote by $A_\omega (\mathbf C_\omega^\perp)$, plays a similar role to that of $\sum_{i=1}^{r}M_i^\perp$. In particular, the closedness of this subspace, or lack thereof, determines the dichotomy between linear and arbitrarily slow convergence. The latter case also implies the super-polynomially fast convergence on some dense linear subspace of $\mathcal H$. Moreover, $A_\omega (\mathbf C_\omega^\perp)$ becomes the set of ``good'' starting points on which we always have at least a polynomial rate of convergence.

It is not difficult to see, that when $r = \infty$, the sets $A_\omega (\mathbf C_\omega^\perp)$ may differ for different sequences of weights $\omega$; see Example \ref{ex:AwCwAreDifferent}. In spite of this, we find that the closedness of $A_\omega (\mathbf C_\omega^\perp)$ in $\mathcal H$ does not depend on the weights $\omega$, but only on the subspaces $M_i$ themselves. To be more precise, we prove that if the set $A_\omega (\mathbf C_\omega^\perp)$ is closed for one sequence of weights $\omega$, then it must be closed for all sequences of weights and be equal to $M^\perp$ (see Theorem \ref{thm:cosCDequal1} in Section \ref{sec:cosCD} phrased in the language of cosines  and Proposition \ref{prop:Minkowski}). Hence it cannot happen that the rate of convergence is linear for one sequence $\omega$, but is arbitrarily slow for another one. This observation slightly strengthens the aforementioned dichotomy theorem.

\section{Preliminaries} \label{sec:preliminaries}
From now on, let
\begin{equation}\label{}
  r \in \mathbb Z_+ \cup \{\infty\}
\end{equation}
and let
\begin{equation}
\Omega_r := \left\{ \{\omega_i\}_{i=1}^r \colon \ \omega_i>0,\ i=1,\ldots,r,\  \sum_{i=1}^{r} \omega_i = 1 \right\}.
\end{equation}
In this section we extend the notation used in the introduction for the particular vector $\omega = \{1/r,\ldots, 1/r\}\in \Omega_r$ to an arbitrary vector $\omega \in \Omega_r$, when $r\in \mathbb Z_+$, and to an arbitrary sequence $\omega \in \Omega_\infty$, when $r=\infty$.

For each $\omega \in \Omega_r$, we define a weighted product $\ell^2$-space and an associated weighted inner product by
\begin{equation}\label{def:Hw}
  \mathbf H_\omega :=
  \begin{cases}
    \mathcal H^r, & \text{if } r\in\mathbb Z_+, \\
    \ell^2_\omega(\mathcal H), & \text{if } r = \infty
  \end{cases}
  \qquad \text{and} \qquad \langle \mathbf x, \mathbf y \rangle_{\omega} := \sum_{i=1}^{r} \omega_i\langle x_i, y_i\rangle,
\end{equation}
where $\ell^2_\omega(\mathcal H):=\{\mathbf x = \{x_i\}_{i=1}^\infty \colon  x_i\in\mathcal H ,\  i=1,2,\ldots,\  \sum_{i=1}^{\infty}\omega_i\|x_i\|^2 <\infty\}$ and where $\mathbf x = \{x_i\}_{i=1}^r, \mathbf y = \{y_i\}_{i=1}^r \in \mathbf H_\omega$. One can verify that the pair $(\mathbf H_\omega, \langle \cdot, \cdot \rangle_{\omega})$ is a Hilbert space. The induced norm on $\mathbf H_\omega$ and the operator norm on $\mathcal B(\mathbf H_\omega)$, the Banach space of all bounded linear operators on $\mathbf H_\omega$, are both denoted by $\|\cdot\|_{\omega}$. Notice that, when $r=\infty$, the weighted $\ell^2$-spaces $\mathbf H_\omega$ may be different for different $\omega \in \Omega_\infty$.

\begin{example}[$r=\infty$]\label{ex:HwAreDifferent}
  Let $\alpha >0$ and $\beta > 1$. Consider the weighted $\ell^2$-space $\mathbf H_{\omega,\beta}=\ell^2_{\omega,\beta}(\mathcal H)$ with weights $\omega_{i,\beta}:=1/(i^\beta s_\beta) $, where $s_\beta:=\sum_{i=1}^{\infty}1/i^\beta$. Moreover, let $\mathbf x_\alpha=\{x_i\}_{i=1}^\infty$ be any sequence with $\|x_i\|=i^{\alpha/2}$. Then $\mathbf x_\alpha \in \mathbf H_{\omega,\beta}$ if and only if $\beta > 1+\alpha$. Consequently, for any fixed $\alpha>0$, we can find $\beta \neq \beta'$  (for example, $1<\beta'<1+\alpha<\beta$) such that $\mathbf x_\alpha \in \mathbf H_{\omega,\beta}$, but $\mathbf x_\alpha \notin \mathbf H_{\omega,\beta'}$.
\end{example}

For each $\omega \in \Omega_r$, we define the \emph{averaging operator} $ A_\omega \colon \mathbf H_\omega \to \mathcal H$ by
\begin{equation}\label{def:A}
  A_\omega(\mathbf x) := \sum_{i=1}^{r}\omega_i x_i,
\end{equation}
where $\mathbf x = \{x_i\}_{i=1}^r \in \mathbf H_\omega$. Note that $A_\omega$ is well defined for $r = \infty$, as the following proposition shows.

\begin{proposition}[$r=\infty$]\label{prop:abs}
  If $\mathbf x = \{x_i\}_{i=1}^\infty\in \mathbf H_\omega$ for some $\omega \in \Omega_\infty$, then the series $\sum_{i=1}^{\infty}\omega_i x_i$ is absolutely convergent, hence unconditionally (compare with the  \hyperref[sec:Appendix]{Appendix}).
\end{proposition}

\begin{proof}
  Observe that $\omega_i \|x_i\| < \omega_i$ for all $i \in I:=\{i\colon \|x_i\|<1\}$ and $\omega_j \|x_j\| \leq \omega_j \|x_j\|^2$ for all $j \in J:=\{j\colon \|x_j\|\geq 1\}$. Consequently, for each $n\geq 1$, we get
  \begin{equation}\label{}
     \sum_{i=1}^{n} \|\omega_i x_i\|
    =  \sum_{\substack{1 \leq i\leq n \\ i\in I}} \omega_i \|x_i\|
    + \sum_{\substack{1 \leq j \leq n\\ j\in J}} \omega_j \|x_j\|
    \leq \sum_{\substack{1 \leq i\leq n \\ i\in I}} \omega_i
    + \sum_{\substack{1 \leq j \leq n\\ j\in J}} \omega_j \|x_j\|^2
    \leq 1 + \|\mathbf x\|_\omega^2 <\infty,
  \end{equation}
  with the convention that the summation over the empty set is zero.
\end{proof}

In particular, the unconditional convergence of $A_\omega(\mathbf x)$ for $r=\infty$ gives us a lot of freedom in rearranging the summands in \eqref{def:A}; see Lemma \ref{lem:rearrangement} in the \hyperref[sec:Appendix]{Appendix}. Moreover, $A_\omega$ is a norm one linear operator which for all $\mathbf x \in \mathbf H_\omega$ and $z\in\mathcal H$ satisfies
\begin{equation}\label{def:Aweak}
  \langle A_\omega(\mathbf x), z \rangle
  = \left\langle \sum_{i=1}^r \omega_i x_i, z \right\rangle
  = \sum_{i=1}^r \omega_i \langle x_i, z \rangle.
\end{equation}

Let $M_i$ be a nontrivial ($M_i\neq \{0\}$) closed and linear subspace of $\mathcal H$, $i=1,\ldots,r$, and let
\begin{equation}\label{def:MinH}
M := \bigcap_{i=1}^{r} M_i.
\end{equation}
For each $\omega \in \Omega_r$, the \emph{simultaneous projection operator} $T_\omega\colon \mathcal H \to \mathcal H$ is defined by
\begin{equation}\label{def:T}
  T_\omega(x) := \sum_{i=1}^{r} \omega_i P_{M_i}(x),
\end{equation}
where $x\in \mathcal H$. Note that $T_\omega (x) = A_\omega(\{P_{M_i}(x)\}_{i=1}^{r})$ and $\sum_{i=1}^{r}\omega_i\|P_{M_i}(x)\|^2 \leq \sum_{i=1}^{r}\omega_i\|x\|^2 = \|x\|^2<\infty$ due to the equalities $\|P_{M_i}\| = 1$, $i=1,\ldots,r$;  see, for example, \cite[Theorem 5.13]{Deutsch2001}.  Hence, for $r=\infty$, the series $T_\omega(x)$ is absolutely convergent by Proposition \ref{prop:abs}. Moreover, since each projection $P_{M_i}$ is self-adjoint,  see again \cite[Theorem 5.13]{Deutsch2001},  this also holds for the simultaneous projection $T_\omega$. Furthermore, by the convexity of $\|\cdot\|$, we have $\|T_\omega\|\leq 1$ and $\fix T_\omega = M$. Indeed, the inclusion $M \subset \fix T_\omega$ is obvious and if there is $x\in \fix T_\omega$ such that $x \notin M_j$ for some $j\in \{1,\ldots,r\}$, then $\|P_{M_j}(x)\| < \|x\|$ and thus we arrive at a contradiction as $\|x\|=\|T_\omega(x)\| \leq \sum_{i=1}^{r}\omega_i \|P_{M_i}(x)\| < \|x\|$.

Following \eqref{int:def:cosM1M2}, for each $\omega \in \Omega_r$, we define the \emph{cosine of the Friedrichs angle} between two nontrivial closed and linear subspaces $\mathbf M_1, \mathbf M_2$ of $\mathbf H_\omega$ by
\begin{equation}\label{def:cosM1N2inHr}
  \cos_\omega(\mathbf M_1, \mathbf M_2)
  := \sup \left\{ \langle \mathbf x, \mathbf y \rangle_\omega \colon
  \begin{array}{l}
    \mathbf x \in \mathbf M_1 \cap (\mathbf M_1 \cap \mathbf M_2)^{\perp_\omega} \cap \mathbf B_\omega,\\
    \mathbf y \in \mathbf M_2 \cap (\mathbf M_1 \cap \mathbf M_2)^{\perp_\omega} \cap \mathbf B_\omega
  \end{array}
  \right\} \in [0,1],
\end{equation}
where $\mathbf B_\omega := \{\mathbf x \in \mathbf H_\omega \colon \|\mathbf x\|_\omega \leq 1\}$. We use the symbol ``$\perp_\omega$'' for the orthogonal complement in $\mathbf H_\omega$ which, for a closed linear subspace $\mathbf M$ of $\mathbf H_\omega$, is defined by
\begin{equation}\label{def:perp}
  \mathbf M^{\perp_\omega} := \{ \mathbf x \in \mathbf H_\omega \colon \langle \mathbf x, \mathbf y \rangle_\omega = 0 \text{ for all } \mathbf y \in \mathbf M\};
\end{equation}
see also Remark \ref{rem:perp}.

We now extend the definition for the product set $\mathbf C_r$ and the diagonal set $\mathbf D_r$ to
\begin{equation}\label{def:CinHw}
  \mathbf C_{\omega} :=
  \begin{cases}
    \prod_{i=1}^{r} M_i, & \text{if } r \in \mathbb Z_+  \\
    \{\{x_i\}_{i=1}^{\infty} \colon x_i \in M_i,\ i=1,2, \ldots,\ \sum_{i=1}^{\infty}\omega_i\| x_i\|^2 < \infty\}, & \text{if } r = \infty.
  \end{cases}
\end{equation}
and
\begin{equation}\label{def:DinHw}
  \mathbf D_{\omega} := \big\{\{x\}_{i=1}^{r} \colon x \in \mathcal H\big\},
\end{equation}
respectively, where $\omega \in \Omega_r$. It is not difficult to see that both $\mathbf C_\omega$ and $\mathbf D_\omega$ are closed and linear subspaces of $\mathbf H_\omega$. In this paper we are interested in the cosine between $\mathbf C_\omega$ and $\mathbf D_\omega$ and its connection to
\begin{equation}\label{def:MinHw}
  \mathbf M_{\omega} :=
  \begin{cases}
    \prod_{i=1}^{r} M, & \text{if } r \in \mathbb Z_+  \\
    \{\{x_i\}_{i=1}^{\infty} \colon x_i \in M,\ i=1,2, \ldots,\ \sum_{i=1}^{\infty}\omega_i\| x_i\|^2 < \infty\}, & \text{if } r = \infty
  \end{cases}
\end{equation}
and
\begin{equation}\label{def:DeltainHw}
  \mathbf \Delta_{\omega} := \prod_{i=1}^{r} B,
\end{equation}
where $B:=\{x\in\mathcal H \colon \|x\|\leq 1\}$. For this reason, we introduce the following \emph{configuration constant}:
\begin{equation}\label{def:cM1N2inHr}
  c_\omega(\mathbf M_1, \mathbf M_2)
  := \sup \left\{ \langle \mathbf x, \mathbf y \rangle_\omega \colon
  \begin{array}{l}
    \mathbf x \in \mathbf M_1 \cap (\mathbf M_1 \cap \mathbf M_2)^{\perp_\omega} \cap \mathbf \Delta_\omega,\\
    \mathbf y \in \mathbf M_2 \cap (\mathbf M_1 \cap \mathbf M_2)^{\perp_\omega} \cap \mathbf \Delta_\omega
  \end{array}
  \right\} \in [0,1].
\end{equation}

\begin{remark}[Space in Question for ``$\perp$''] \label{rem:perp}
  In our further study, by fixing $\omega \in \Omega_r$, we restrict our analysis only to the Hilbert space $(\mathbf H_\omega, \langle \cdot, \cdot \rangle_{\omega})$. In particular, when $r = \infty$, we consider the space $(\ell^2_\omega(\mathcal H), \langle \cdot, \cdot \rangle_{\omega})$. Therefore, by using the subscript ``$\omega$'' added to a set, we implicitly assume that such a set is considered in $(\mathbf H_\omega, \langle \cdot, \cdot \rangle_{\omega})$. Knowing the underlying space becomes very important especially for the operation of the orthogonal complement. For example, the diagonal set $\mathbf D_\omega$ is the same for all $\omega\in \Omega_r$ while its orthogonal complement $(\mathbf D_\omega)^{\perp_\omega}$ may be different for different $\omega$'s (see Proposition \ref{prop:Minkowski}). Having this in mind, we may simply write ``$\perp$'' instead of ``$\perp_\omega$'' when the underlying space $\mathbf H_\omega$ is known from the context. For example, $\mathbf D_\omega^{\perp}$ reads as $(\mathbf D_\omega)^{\perp_\omega}$.
\end{remark}

We proceed with the following proposition.

\begin{proposition} \label{prop:inclusion}
Let $\omega \in \Omega_r$. We have
\begin{equation}\label{prop:inclusion:Mperp}
  \mathbf M_{\omega}^\perp =
  \begin{cases}
    \prod_{i=1}^{r} M^\perp, & \text{if } r \in \mathbb Z_+  \\
    \{\{x_i\}_{i=1}^{\infty} \colon x_i \in M^\perp,\ i=1,2, \ldots,\ \sum_{i=1}^{\infty}\omega_i\| x_i\|^2 < \infty\}, & \text{if } r = \infty.
  \end{cases}
\end{equation}
and $\mathbf D_{\omega} \cap \mathbf M_{\omega}^\perp = \mathbf D_{\omega} \cap (\mathbf C_{\omega} \cap \mathbf D_{\omega})^\perp.$   However, the inclusion
\begin{equation} \label{prop:inclusion:eq}
  \mathbf C_{\omega} \cap \mathbf M_{\omega}^\perp \subset \mathbf C_{\omega} \cap (\mathbf C_{\omega} \cap \mathbf D_{\omega})^\perp
\end{equation}
is strict if $M\neq \{0\}$ and $M_j \neq M$ for some $j \in \{1,\ldots,r\}$.
\end{proposition}

\begin{proof}
  We begin by showing the first equality. We denote the set on the right-hand side of \eqref{prop:inclusion:Mperp} by $\mathbf M_\omega'$. It is not difficult to see that $\mathbf M_{\omega}^\perp \supset \mathbf M_\omega'$. We now demonstrate the opposite inclusion ``$\subset$''. Indeed, let $\mathbf x = \{x_i\}_{i=1}^r \in \mathbf M_\omega^\perp$. Then, by \eqref{def:perp}, $\langle \mathbf x, \mathbf y \rangle_\omega = 0$ for all $\mathbf y = \{y_i\}_{i=1}^r \in \mathbf M_\omega$, where $y_i \in M$. By choosing $j\in \{1,\ldots,r\}$ and $y_i:=0$ for all $i\neq j$, we obtain $\langle x_j, y_j\rangle = 0$ for all $y_j \in M$, and thus we must have $x_j \in M^\perp$. The arbitrariness of $j$ implies that  $x_j \in M^\perp$ for all $j\in\{1,\ldots,r\}$. Note that up to now, the above-presented argument holds for $r \in \mathbb Z_+$ as well as for $r = \infty$.  In the latter case ($r = \infty$), by \eqref{def:perp}, we see that, $\mathbf x \in \mathbf H_\omega$.  Consequently, $\sum_{i=1}^{\infty}\omega_i \|x_i\|^2 < \infty$.  This shows that $\mathbf x \in \mathbf M_\omega'$, as asserted.

  Next we show the second equality  and inclusion \eqref{prop:inclusion:eq}. To this end, observe that, by the definition of $\mathbf C_\omega$ and $\mathbf D_\omega$, we get
  \begin{equation}\label{pr:inclusion:CD}
    \mathbf C_{\omega} \cap \mathbf D_{\omega}
    = \left\{\{x\}_{i=1}^r \colon x \in M \right\}.
  \end{equation}
  Consequently, by using the first equality, we see that
  \begin{align}\label{} \nonumber
    \mathbf D_\omega \cap \mathbf M_\omega^\perp
    & = \left\{ \{x\}_{i=1}^r \colon x \in M^\perp \right\} \\ \nonumber
    & = \left\{ \{x\}_{i=1}^r \colon \langle x, y \rangle = 0 \text{ for all } y \in M \right\} \\ \nonumber
    & = \left\{ \mathbf x \in \mathbf D_\omega \colon \langle \mathbf x, \mathbf y \rangle_\omega = 0 \text{ for all } \mathbf y \in \mathbf C_{\omega} \cap \mathbf D_{\omega} \right\} \\
    & = \mathbf D_\omega \cap (\mathbf C_{\omega} \cap \mathbf D_{\omega})^\perp
  \end{align}
  and
  \begin{align}\label{} \nonumber
    \mathbf C_{\omega} \cap \mathbf M_{\omega}^\perp
    & = \left\{\{x_i\}_{i=1}^r \in \mathbf H_\omega \colon x_i \in M_i\cap M^\perp,\ i=1,\ldots,r \right\} \\ \nonumber
    & \subset \left\{\{x_i\}_{i=1}^r \in \mathbf H_\omega \colon x_i \in M_i,\ i=1,\ldots,r,\  \sum\nolimits_{i=1}^r \omega_i x_i \in  M^\perp \right\} \\
    & = \mathbf C_{\omega} \cap (\mathbf C_{\omega} \cap \mathbf D_{\omega})^\perp.
  \end{align}

  Finally, we show that inclusion \eqref{prop:inclusion:eq} is strict.  The assumption $M_j \neq M$ guarantees that the subspace $M_j \cap M^\perp$ is nontrivial. Indeed, by the orthogonal decomposition theorem (that is, $I = P_{M} + P_{M^\perp}$), for any $x\in M_j \setminus M$, the point $x_j:= P_{M^\perp}(x) = x - P_{M}(x) \neq 0 $ and $x_j \in M_j \cap M^\perp$.

  Let $m\in M$ be nonzero. Define $\mathbf y = \{y_i\}_{i=1}^r$ by $y_j:=\frac 1 {\omega_j} (m+x_j)$, $y_{j+1}:= - \frac 1 {\omega_{j+1}} m$ and $y_i:=0$ for all $i\neq j,j+1$.  Note that $y_j \notin M^\perp$. Thus $\mathbf y \notin \mathbf C_{\omega} \cap \mathbf M_{\omega}^\perp$.  On the other hand, $y_i \in M_i$ for all $i=1,\ldots,r,$ and, moreover, for any $m'\in M$, we have
  \begin{equation}
   \sum_{i=1}^{r} \omega_i \langle  y_i, m'\rangle
   = \langle (m+x_j) - m, m'\rangle
   = \langle x_j, m'\rangle = 0,
  \end{equation}
  which proves that $\mathbf y \in \mathbf C_{\omega} \cap (\mathbf C_{\omega} \cap \mathbf D_{\omega})^\perp$.
\end{proof}

In the next proposition we show the connections among the sets $\mathbf C_{\omega}, \mathbf D_{\omega}$ and the averaging operator $A_\omega$. When $r \in \mathbb Z_+$, equality \eqref{prop:Minkowski:M} can be found in \cite[Theorem 4.6 (5)]{Deutsch2001}.

\begin{proposition}\label{prop:Minkowski}
  Let $\omega \in \Omega_r$. We have
  \begin{equation}
  \mathbf C_{\omega}^\perp =
  \begin{cases}
    \prod_{i=1}^{r} M_i^\perp, & \text{if } r \in \mathbb Z_+  \\
    \{\{x_i\}_{i=1}^{\infty} \colon x_i \in M_i^\perp,\ i=1,2, \ldots,\ \sum_{i=1}^{\infty}\omega_i \|x_i\|^2 < \infty\}, & \text{if } r = \infty
  \end{cases}
  \end{equation}
  and $\mathbf D_\omega^\perp = \mathcal N (A_\omega)$ -- the null space of $A_\omega$. Consequently, the subspace $\mathbf C_\omega^\perp + \mathbf D_\omega^\perp$ is closed in $\mathbf H_\omega$ if and only if the subspace
  \begin{equation}\label{prop:Minkowski:Aw}
  A_\omega (\mathbf C_\omega^\perp) =
  \begin{cases}
    \sum_{i=1}^{r} M_i^\perp, & \text{if } r\in \mathbb Z_+,\\
    \big\{\sum_{i=1}^{\infty} \omega_i x_i \colon x_i \in M_i^\perp,\ i=1,2,\ldots,\ \sum_{i=1}^{\infty}\omega_i\|x_i\|^2<\infty \big\}, & \text{if } r= \infty
  \end{cases}
  \end{equation}
  is closed in $\mathcal H$. Moreover,
  \begin{equation}\label{prop:Minkowski:M}
    \overline{A_\omega (\mathbf C_\omega^\perp)} = M^\perp.
  \end{equation}
\end{proposition}

\begin{proof}
The first equality can be established by using an argument similar to the one presented in the proof of Proposition \ref{prop:inclusion} with $M$ replaced by $M_j$.

In order to show the second equality, take $\mathbf x \in \mathbf D_\omega^\perp$. Then, by \eqref{def:Aweak}, for all $\mathbf y = \{y\}_{i=1}^r \in \mathbf D_\omega$, we have $\langle \mathbf x, \mathbf y\rangle_\omega = \langle A_\omega (\mathbf x), y\rangle = 0$. In particular, by taking $y:=A_\omega(\mathbf x)$, we see that $A_\omega(\mathbf x) = 0$, that is, $\mathbf x \in \mathcal N(A_\omega)$. On the other hand, it is easy to see that when $\mathbf x \in \mathcal N(A_\omega)$, then for all $\mathbf y = \{y\}_{i=1}^r \in \mathbf D_\omega$, we have $0 = \langle A_\omega(\mathbf x), y\rangle = \langle \mathbf x, \mathbf y \rangle_\omega$, that is, $\mathbf x \in \mathbf D_\omega^\perp$. This shows the second equality.

Recall that $ A_\omega$ is linear and bounded. In view of
\cite[section 17H, p. 142]{Holmes1975},  the set $A_\omega(\mathbf C_\omega^\perp)$ is closed in $\mathcal H$ if and only if $\mathbf C_\omega^\perp + \mathcal N(A_\omega)$ is closed in $\mathbf H_\omega$. The equalities in \eqref{prop:Minkowski:Aw} follow from the above discussion.

We now focus on \eqref{prop:Minkowski:M}, where we first show that
\begin{equation}\label{pr:Minkowski:M}
  \big(A_\omega (\mathbf C_\omega^\perp)\big)^\perp = M.
\end{equation}
It is not difficult to see that $M \subset \big(A_\omega (\mathbf C_\omega^\perp)\big)^\perp$. In order to demonstrate the opposite inclusion ``$\supset$'', take $x \in \big(A_\omega (\mathbf C_\omega^\perp)\big)^\perp$. Consequently, for all $\mathbf y = \{y_i\}_{i=1}^r \in \mathbf C_\omega^\perp$ (hence $y_i \in M_i^\perp$), we have $\langle x, A_\omega(\mathbf y)\rangle =0$. In particular, by choosing $j \in \{1,\ldots,r\}$ and setting $y_i:=0$ for all $i\neq j$, we obtain that $\langle x, y_j\rangle = 0$ for all $y_j \in M_j^\perp$. This, when combined with the fact that $M_j$ is a closed linear subspace, implies that $x \in M_j^{\perp \perp} = M_j$; see, for example, \cite[Theorem 4.5 (8)]{Deutsch2001}. The arbitrariness of $j\in\{1,\ldots,r\}$ yields that $x \in M$.

We now return to \eqref{prop:Minkowski:M}. Recall that, when $L$ is a linear subspace of $\mathcal H$ which is not necessarily closed, then $L^\perp = (\overline L)^\perp$; see \cite[Theorem 4.5 (2)]{Deutsch2001}. Consequently, by \eqref{pr:Minkowski:M},
\begin{equation}\label{}
  M^\perp = \big(A_\omega (\mathbf C_\omega^\perp)\big)^{\perp \perp}
  = \big(\overline{A_\omega (\mathbf C_\omega^\perp)}\big)^{\perp \perp}
  = \overline{A_\omega (\mathbf C_\omega^\perp)},
\end{equation}
which completes the proof.
\end{proof}

Note that similarly to Example \ref{ex:HwAreDifferent}, when $r = \infty$, the sets $A_\omega(\mathbf C_\omega^\perp)$ may be different for different $\omega \in \Omega_\infty$.

\begin{example}[$r=\infty$]\label{ex:AwCwAreDifferent}
  Assume that $\mathcal H$ is separable and let $\{e_i\}_{i=1}^\infty$ be a norm-one Schauder basis of it. Let $M_i^\perp:=\operatorname{span}\{e_i\}$.  Using the notation of Example \ref{ex:HwAreDifferent},  consider two spaces, $\mathbf H_{\omega,\beta}$ and $\mathbf H_{\omega,\beta'}$, with
  \begin{equation}\label{ex:AwCwAreDifferent:alhpaBeta}
    \varepsilon >0,\quad \alpha >0, \quad \beta:=1+\alpha+\varepsilon \quad \text{and} \quad
    \alpha':=\alpha+2\varepsilon, \quad \beta':= 1+\alpha'.
  \end{equation}
  By using Example \ref{ex:HwAreDifferent}, we see that
  \begin{equation}\label{ex:AwCwAreDifferent:xx}
    \mathbf x_\alpha :=\{i^{\alpha/2}e_i\}_{i=1}^\infty \in \mathbf H_{\omega,\beta}
    \quad \text{and} \quad
    \mathbf x_{\alpha'} :=\{i^{\alpha'/2}e_i\}_{i=1}^\infty \notin \mathbf H_{\omega,\beta'}.
  \end{equation}
  Consequently, $y := \frac{s_{\beta}}{s_{\beta'}} A_{\omega,\beta}(\mathbf x_\alpha) \in A_{\omega,\beta}(\mathbf C_{\omega,\beta}^\perp)$. Note that, by the choice of the $e_i$'s, the representation of
  \begin{equation}\label{}
    y = \frac{1}{s_{\beta'}} \sum_{i=1}^{\infty} \frac{1}{i^{\beta'}} \left( i^{[\alpha+2(\beta'-\beta)]/2} e_i\right)
    = \frac{1}{s_{\beta'}} \sum_{i=1}^{\infty} \frac{1}{i^{\beta'}} \left( i^{\alpha'/2} e_i\right)
  \end{equation}
  is unique. Hence, by \eqref{ex:AwCwAreDifferent:xx}, we see that $y \notin A_{\omega,\beta'}(\mathbf C_{\omega,\beta'}^\perp)$.
\end{example}

\begin{remark}[Notation]\label{rem:notation}
To emphasize that we refer to a particular vector $\omega = \{1/r,\ldots,1/r\} \in \Omega_r$ for some $r \in \mathbb Z_+$, we  may replace  the subscript ``$\omega$'' by the subscript ``$r$'' in all the above-mentioned definitions. For example, we write
\begin{equation}
\mathbf H_r,\ \langle \cdot, \cdot \rangle_r,\ \|\cdot\|_r, \quad
\mathbf C_r,\ \mathbf D_r,\ \mathbf M_r,\ \mathbf \Delta_r, \quad
\cos_r(\cdot, \cdot),\ c_r(\cdot, \cdot), \quad
T_r \text{ and } A_r
\end{equation}
instead of
\begin{equation}
\mathbf H_\omega,\ \langle \cdot, \cdot \rangle_\omega,\ \|\cdot\|_\omega, \quad
\mathbf C_\omega,\ \mathbf D_\omega,\ \mathbf M_\omega,\ \mathbf \Delta_\omega, \quad
\cos_\omega(\cdot, \cdot),\ c_\omega(\cdot, \cdot), \quad
T_\omega \text{ and } A_\omega,
\end{equation}
respectively. This coincides with the notation used in the introduction.
\end{remark}


\section{Alternating Projection Formalization of Pierra} \label{sec:Pierra}
In the next two results we bring out the connections among the operators $A_\omega$, $P_M$ and $T_\omega$, and the projections $P_{\mathbf C_\omega}$, $P_{\mathbf D_\omega}$ and $P_{\mathbf C_\omega \cap \mathbf D_\omega}$.

\begin{lemma}\label{lem:ProjCD}
  Let $\omega \in \Omega_r$. For each $\mathbf x := \{x_i\}_{i=1}^r \in \mathbf H_\omega$, we have
  \begin{equation}\label{lem:ProjCD:C}
    P_{\mathbf C_{\omega}}(\mathbf x) = \left\{P_{M_i}(x_i) \right\}_{i=1}^r,
  \end{equation}
  \begin{equation}\label{lem:ProjCD:D}
    P_{\mathbf D_\omega}(\mathbf x) = \left\{A_\omega(\mathbf x) \right\}_{i=1}^r
  \end{equation}
  and
  \begin{equation}\label{lem:ProjCD:CD}
    P_{\mathbf C_{\omega} \cap \mathbf D_\omega}(\mathbf x) = \left\{ P_{M}(A_\omega(\mathbf x)) \right\}_{i=1}^r.
  \end{equation}
\end{lemma}

\begin{proof}
  Recall that for a closed and  linear subspace $L$ of $\mathcal H$, we have
  \begin{equation}\label{pr:ProjCD:PL1}
    y = P_{L}( x)
    \quad \Longleftrightarrow \quad
     y \in  L \quad\text{and}\quad \langle  x -  y,  z\rangle = 0 \quad \forall z\in L;
  \end{equation}
  see, for example, \cite[Theorem 4.9]{Deutsch2001}. Analogously, for a closed and linear subspace $\mathbf L$ of $\mathbf H_\omega$, we have
  \begin{equation}\label{pr:ProjCD:PL2}
    \mathbf y = P_{\mathbf L}(\mathbf x)
    \quad \Longleftrightarrow \quad
    \mathbf y \in \mathbf L \quad\text{and}\quad \langle \mathbf x - \mathbf y, \mathbf z\rangle_\omega = 0 \quad \forall \mathbf z\in \mathbf L.
  \end{equation}
  We now consider each asserted equality separately.

  By definition, the point $\mathbf y:= \{P_{M_i}(x_i)\}_{i=1}^r$ satisfies $P_{M_i}(x_i) \in M_i$, $i = 1,\ldots,r$. Moreover, using the equality $\|P_{M_i}\|=1$, we see that
  \begin{equation}
  \sum_{i=1}^{r}\omega_i \|P_{M_i}(x_i)\|^2
  \leq \sum_{i=1}^{r}\omega_i \|x_i\|^2
  = \|\mathbf x\|^2_\omega <\infty.
  \end{equation}
  This shows that $\mathbf y \in \mathbf C_\omega$. Furthermore, by \eqref{pr:ProjCD:PL1} applied to $L := M_i$, $i=1,\ldots,r$, for each $\mathbf z = \{z_i\}_{i=1}^r \in \mathbf C_\omega$, we have
  \begin{equation}\label{}
    \langle \mathbf x - \mathbf y, \mathbf z\rangle_{\omega}
    = \sum_{i=1}^{r}\omega_i \langle  x_i -  P_{M_i}(x_i),  z_i\rangle = 0.
  \end{equation}
  By \eqref{pr:ProjCD:PL2}, this shows \eqref{lem:ProjCD:C}.

  By definition, $\mathbf y:= \{A_\omega(\mathbf x)\}_{i=1}^r \in \mathbf D_\omega$. Moreover, by \eqref{def:Aweak}, for each $\mathbf z = \{z\}_{i=1}^r \in \mathbf D_\omega$, we have
  \begin{equation}\label{}
    \langle \mathbf x - \mathbf y, \mathbf z\rangle_{\omega}
    = \sum_{i=1}^{r}\omega_i \langle  x_i - A_\omega(\mathbf x),  z\rangle
    = \left\langle \sum_{i=1}^{r}\omega_i( x_i - A_\omega(\mathbf x)),  z \right\rangle
    = \langle  A_\omega(\mathbf x) - A_\omega(\mathbf x),  z\rangle = 0.
  \end{equation}
  Again, by \eqref{pr:ProjCD:PL2}, this proves \eqref{lem:ProjCD:D}.

  Finally, let now $\mathbf y:= \{P_{M}(A_\omega(\mathbf x))\}_{i=1}^r$. It is clear that, by definition, $\mathbf y \in \mathbf C_\omega \cap \mathbf D_\omega = \{\mathbf x = \{x\}_{i=1}^r \colon x \in M\}.$ By \eqref{def:Aweak} and \eqref{pr:ProjCD:PL1}, for any $\mathbf z = \{z\}_{i=1}^r \in \mathbf C_\omega \cap \mathbf D_\omega$, we have
  \begin{align}\label{} \nonumber
    \langle \mathbf x - \mathbf y, \mathbf z\rangle_{\omega}
    & = \sum_{i=1}^{r} \omega_i \langle  x_i - P_{M}(A_\omega(\mathbf x)),  z\rangle
    =  \left\langle  \sum_{i=1}^{r} \omega_i \Big(x_i - P_{M}(A_\omega(\mathbf x)) \Big),  z \right\rangle \\
    & = \langle   A_\omega(\mathbf x)  - P_{M}(A_\omega(\mathbf x)),  z\rangle = 0.
  \end{align}
  This, in view of \eqref{pr:ProjCD:PL2}, proves the last equality.
\end{proof}

\begin{theorem}\label{thm:normConvergence}
  Let $\omega \in \Omega_r$. For each $x \in \mathcal H$ and $\mathbf x := \{x\}_{i=1}^r$, we have
  \begin{equation}\label{thm:normConvergence:eq}
    \|T_\omega^k(x) - P_{M}(x)\|
    =\|(P_{\mathbf D_\omega} P_{\mathbf C_\omega})^k(\mathbf x) - P_{\mathbf C_\omega \cap \mathbf D_\omega}(\mathbf x)\|_\omega
    \to 0 \text{\quad as } k \to \infty.
  \end{equation}
\end{theorem}
\begin{proof}
  Using \eqref{lem:ProjCD:C}, \eqref{lem:ProjCD:D} and induction with respect to $k$, we see that the equality
  \begin{equation}\label{}
    (P_{\mathbf D_\omega}P_{\mathbf C_\omega})^k(\mathbf x) = \{T_\omega^k(x)\}_{i=1}^r
  \end{equation}
  holds for all $x \in \mathcal H$ and $\mathbf x:=\{x\}_{i=1}^r \in \mathbf H_\omega$. This, when combined with \eqref{lem:ProjCD:CD}, leads to
  \begin{equation}\label{}
    \|T_\omega^k(x)-P_M(x)\|
    = \|\{T_\omega^k(x)\}_{i=1}^r - \{P_M(x)\}_{i=1}^r\|_\omega
    = \|(P_{\mathbf D_\omega} P_{\mathbf C_\omega})^k(\mathbf x) - P_{\mathbf C_\omega \cap \mathbf D_\omega}(\mathbf x)\|_\omega,
  \end{equation}
  which proves the equality in \eqref{thm:normConvergence:eq}.  The norm convergence of the alternating projection method follows from Theorem \ref{int:th:norm} applied to $M_1:=\mathbf C_\omega$ and $M_2:=\mathbf D_\omega$ in $\mathbf H_\omega$.
\end{proof}

\begin{theorem}[Exact Norm Value]\label{thm:norm}
  Let $\omega \in \Omega_r$. For each $k=1,2,\ldots,$ we have
  \begin{align}\label{thm:norm:eq}
    \|T_\omega^k - P_{M}\|
    = \|(P_{\mathbf D_{\omega}}P_{\mathbf C_{\omega}} P_{\mathbf D_{\omega}})^k - P_{\mathbf C_{\omega} \cap \mathbf D_{\omega}}\|_\omega
    = \cos_{\omega}(\mathbf C_{\omega}, \mathbf D_{\omega})^{2k}
     \leq 1.
  \end{align}
\end{theorem}

\begin{proof}
  The proof follows the argument in \cite[Theorem 7]{ReichZalas2017}. We give it here for the convenience of the reader.

   Recall that the operator $T_\omega$ is self-adjoint, $\|T_\omega\|\leq 1$  and $\fix T_\omega = M$ (compare with Section \ref{sec:preliminaries}). Moreover, for each $i=1,\ldots,r$, the projection $P_{M_i}$ commutes with $P_M$, that is,
  \begin{equation}\label{pr:norm:PM_PMi}
    P_{M} P_{M_i} = P_{M_i}P_{M} = P_{M};
  \end{equation}
  see \cite[Lemma 9.2]{Deutsch2001}. Consequently, the operator $T_\omega$ commutes with $P_M$ too and we have
  \begin{equation}\label{}
    P_M T_\omega = T_\omega P_M = P_M.
  \end{equation}
  By using  \cite[Lemma 6]{ReichZalas2017}, we get $\|T_\omega^k-P_M\| = \|T_\omega-P_M\|^k$.

  On the other hand, the operator $\mathbf T:=P_{\mathbf D_{\omega}}P_{\mathbf C_{\omega}} P_{\mathbf D_{\omega}}$ is also self-adjoint, $\|\mathbf T\|_\omega\leq 1$ and $\fix \mathbf T = \mathbf C_{\omega} \cap \mathbf D_{\omega}$. Note that similarly to \eqref{pr:norm:PM_PMi}, the projections $P_{\mathbf C_{\omega}}$ and $P_{\mathbf D_{\omega}}$ commute with $P_{\mathbf C_{\omega} \cap \mathbf D_{\omega}}$, where
  \begin{equation}\label{pr:norm:PCD_PC_PD}
    P_{\mathbf C_\omega \cap \mathbf D_\omega}
    = P_{\mathbf C_\omega \cap \mathbf D_\omega} P_{\mathbf C_\omega}
    = P_{\mathbf C_\omega} P_{\mathbf C_\omega \cap \mathbf D_\omega}
    = P_{\mathbf C_\omega \cap \mathbf D_\omega} P_{\mathbf D_\omega}
    = P_{\mathbf D_\omega} P_{\mathbf C_\omega \cap \mathbf D_\omega}.
  \end{equation}
  This leads to
  \begin{equation}\label{pr:norm:PM_T}
    P_{\mathbf C_{\omega} \cap \mathbf D_{\omega}} \mathbf T
    = \mathbf T P_{\mathbf C_{\omega} \cap \mathbf D_{\omega}}
    = P_{\mathbf C_{\omega} \cap \mathbf D_{\omega}}.
  \end{equation}
   Again, by using \cite[Lemma 6]{ReichZalas2017},  but this time in $\mathbf H_\omega$,  we obtain $\|\mathbf T^k - P_{\mathbf C_{\omega} \cap \mathbf D_{\omega}}\|_\omega = \|\mathbf T - P_{\mathbf C_{\omega} \cap \mathbf D_{\omega}}\|_\omega^k$.

  In order to complete the proof, it suffices to show the equalities of \eqref{thm:norm:eq} only for $k=1$. By the properties of the adjoint operation ``$*$'' and by Theorem \ref{int:th:APM}, we obtain
  \begin{align}\label{pr:norm:1}
  \nonumber
    \|P_{\mathbf D_\omega}P_{\mathbf C_\omega}P_{\mathbf D_\omega}-P_{\mathbf C_\omega\cap\mathbf D_\omega}\|_{ \omega}
    &=\|P_{\mathbf D_\omega}P_{\mathbf C_\omega}P_{\mathbf C_\omega}P_{\mathbf D_\omega}-P_{\mathbf C_\omega\cap\mathbf D_\omega}\|_{ \omega} \\
    \nonumber
    &=\|(P_{\mathbf D_\omega}P_{\mathbf C_\omega}-P_{\mathbf C_\omega\cap\mathbf D_\omega})
    (P_{\mathbf C_\omega}P_{\mathbf D_\omega}-P_{\mathbf C_\omega\cap\mathbf D_\omega})\|_{ \omega} \\
    \nonumber
    &=\|(P_{\mathbf D_\omega}P_{\mathbf C_\omega}-P_{\mathbf C_\omega\cap\mathbf D_\omega})
    (P_{\mathbf D_\omega}P_{\mathbf C_\omega}-P_{\mathbf C_\omega\cap\mathbf D_\omega})^*\|_{ \omega} \\
    \nonumber
    &=\|P_{\mathbf D_\omega}P_{\mathbf C_\omega}-P_{\mathbf C_\omega\cap\mathbf D_\omega}\|^2_{ \omega} \\
    &= \cos_\omega(\mathbf C_\omega, \mathbf D_\omega)^2.
  \end{align}
  Let $\mathbf B_\omega:= \{\mathbf x \colon \|\mathbf x\|_\omega \leq 1\}$. Since $P_{\mathbf D_\omega}(\mathbf B_\omega)= \mathbf D_\omega \cap \mathbf B_\omega$, we see that
  \begin{align}\label{pr:norm:2} \nonumber
  \|P_{\mathbf D_\omega}P_{\mathbf C_\omega}P_{\mathbf D_\omega}-P_{\mathbf C_\omega\cap\mathbf D_\omega}\|_{ \omega}
  & =\|P_{\mathbf D_\omega}P_{\mathbf C_\omega}P_{\mathbf D_\omega} -  P_{\mathbf C_\omega\cap\mathbf D_\omega}  P_{\mathbf D_\omega}\|_{ \omega} \\ \nonumber
  & =\sup\left\{ \|P_{\mathbf D_\omega}P_{\mathbf C_\omega}P_{\mathbf D_\omega}(\mathbf x)-P_{\mathbf C_\omega\cap\mathbf D_\omega}P_{\mathbf D_\omega}(\mathbf x)\|_{ \omega} \colon \mathbf x \in \mathbf B_\omega \right\} \\ \nonumber
  & = \sup\left\{ \|P_{\mathbf D_\omega}P_{\mathbf C_\omega}(\mathbf y)-P_{\mathbf C_\omega\cap\mathbf D_\omega}(\mathbf y)\|_{ \omega} \colon \mathbf y \in P_{\mathbf D_\omega}(\mathbf B_\omega) \right\}\\ \nonumber
  & = \sup\left\{ \|P_{\mathbf D_\omega}P_{\mathbf C_\omega}(\mathbf y)-P_{\mathbf C_\omega\cap\mathbf D_\omega}(\mathbf y)\|_{ \omega} \colon \mathbf y \in \mathbf D_\omega \cap \mathbf B_\omega \right\}\\ \nonumber
  & = \sup\left\{ \|T_\omega(y)-P_M(y)\| \colon y \in \mathcal H \text{ and } \|y\|\leq 1 \right\}\\
  & = \|T_\omega-P_M\|.
\end{align}
This completes the proof.
\end{proof}

Note that equality \eqref{thm:norm:eq} becomes useful only when  $\cos_\omega(\mathbf C_\omega, \mathbf D_\omega) <1$ in which case it turns into the optimal error bound. We return to this inequality in Theorems \ref{thm:cosCDequal1} and \ref{thm:equiv} below.

\section[Properties of the Cosine]{Properties of the Cosine $\cos_\omega(\mathbf C_\omega, \mathbf D_\omega)$ \label{sec:cosCD}}

In the next two lemmata, we show that the set $(\mathbf C_\omega \cap \mathbf D_\omega)^\perp$ can be replaced by its subset $\mathbf M_\omega^\perp$ in the definitions of $\cos_\omega(\mathbf C_\omega, \mathbf D_\omega)$ and $c_\omega(\mathbf C_\omega, \mathbf D_\omega)$, despite the discussion concerning the inclusion of Proposition \ref{prop:inclusion}.

\begin{lemma} \label{lem:cosCD1}
Let $\omega \in \Omega_r$.  The cosine of the Friedrichs angle between $\mathbf C_\omega$ and $\mathbf D_\omega$ satisfies:
\begin{align} \nonumber \label{lem:cosCD1:sup}
  \cos_\omega(\mathbf C_\omega, \mathbf D_\omega)
  & = \sup \left\{ \langle \mathbf x, \mathbf y \rangle_\omega \colon
  \begin{array}{l}
    \mathbf x \in \mathbf C_\omega \cap \mathbf M_\omega^\perp \cap \mathbf B_\omega,\\
    \mathbf y \in \mathbf D_\omega \cap \mathbf M_\omega^\perp \cap \mathbf B_\omega
  \end{array}
  \right\}\\
  & = \sup
  \left\{
    \sum_{i=1}^{r} \omega_i \langle x_i, y\rangle \colon
    \begin{array}{l}
      x_i \in M_i \cap M^\perp,\ i=1,\ldots,r,\ \sum_{i=1}^{r}\omega_i\| x_i\|^2 \leq 1,\\
      y \in M^\perp,\ \|y\| \leq 1
    \end{array}
  \right\}.
\end{align}
\end{lemma}

\begin{proof}
Denote the right-hand side of \eqref{lem:cosCD1:sup} by $\alpha$ and observe that, by the inclusion $\mathbf C_{\omega} \cap \mathbf M_{\omega}^\perp \subset \mathbf C_{\omega} \cap (\mathbf C_{\omega} \cap \mathbf D_{\omega})^\perp$, we have $\cos_\omega(\mathbf C_\omega, \mathbf D_\omega) \geq \alpha$. We now show that $\cos_\omega(\mathbf C_\omega, \mathbf D_\omega) \leq \alpha$.

Note first that, analogously to \eqref{pr:norm:PM_PMi}, for each $i=1,\ldots,r$, we have
\begin{equation}\label{pr:cosCD1:PM_PMi:perp}
  P_{M^\perp} P_{M_i} = P_{M_i}P_{M^\perp} = P_{M_i\cap M^\perp}.
\end{equation}
Indeed, when $x \in M_i$, by \eqref{pr:norm:PM_PMi} and using the orthogonal decomposition theorem, we see that
\begin{equation}\label{}
  P_{M^\perp} (x) = P_{M^\perp} P_{M_i}(x) =  P_{M_i}(x) - P_{M} P_{M_i}(x) = P_{M_i}(x) - P_{M}(x) \in M_i,
\end{equation}
that is, $P_{M^\perp}(M_i)\subset M_i$. Hence we may again apply \cite[Lemma 9.2]{Deutsch2001} to obtain \eqref{pr:cosCD1:PM_PMi:perp}.

Let $\mathbf x = \{x_i\}_{i=1}^r \in \mathbf H_\omega$ be such that $x_i \in M_i$, $i=1,\ldots,r$, and let $y\in M^\perp$. Using \eqref{pr:cosCD1:PM_PMi:perp}, we arrive at
\begin{equation}\label{pr:cosCD1:reduction}
  \langle x_i, y\rangle = \langle P_{M}(x_i)+P_{M^\perp}(x_i), y\rangle
  = \langle P_{M^\perp}P_{M_i}(x_i), y\rangle
  = \langle P_{M_i \cap M^\perp}(x_i), y\rangle.
\end{equation}
Furthermore, by \eqref{pr:norm:PM_PMi} and \eqref{pr:cosCD1:PM_PMi:perp}, we obtain
\begin{equation}\label{}
  \|x_i\|^2 = \|P_{M_i\cap M}(x_i)\|^2
  + \|P_{M_i\cap M^\perp}(x_i)\|^2.
\end{equation}
Therefore,
\begin{align}\label{pr:cosCD1:cosCD}\nonumber
  \cos_\omega(\mathbf C_\omega, \mathbf D_\omega)
  & = \sup\left\{
    \sum_{i=1}^r \omega_i \langle x_i, y\rangle \colon
    \begin{array}{l}
      x_i \in M_i,\ i=1,\ldots,r,\ \sum_{i=1}^{r}\omega_i x_i \in M^\perp,\\
      \sum_{i=1}^{r}\omega_i\| x_i\|^2 \leq 1,\ y \in M^\perp,\ \|y\| \leq 1
    \end{array}
    \right\}\\ \nonumber
  & \leq \sup\left\{
    \sum_{i=1}^r \omega_i \langle x_i, y\rangle \colon
    \begin{array}{l}
      x_i \in M_i,\ i=1,\ldots,r,\ \sum_{i=1}^{r}\omega_i\| x_i\|^2 \leq 1,\\
      y \in M^\perp,\ \|y\| \leq 1
    \end{array}
    \right\}\\ \nonumber
  & \leq \sup\left\{
    \sum_{i=1}^r \omega_i \langle x_i, y\rangle \colon
    \begin{array}{l}
      x_i \in M_i,\ i=1,\ldots,r,\ \sum_{i=1}^r \omega_i \|P_{M_i\cap M^\perp}(x_i)\|^2\leq 1,\\
      y \in M^\perp,\ \|y\| \leq 1
    \end{array}
    \right\}\\ \nonumber
  & = \sup\left\{
    \sum_{i=1}^r \omega_i \langle z_i, y\rangle \colon
    \begin{array}{l}
      z_i \in M_i\cap M^\perp,\ i=1,\ldots,r,\  \sum_{i=1}^{r}\omega_i\| z_i\|^2 \leq 1,\\
      y \in M^\perp,\ \|y\| \leq 1
    \end{array}
    \right\}\\
  & = \sup \left\{ \langle \mathbf z, \mathbf y \rangle_\omega \colon
  \begin{array}{l}
    \mathbf z \in \mathbf C_\omega \cap \mathbf M_\omega^\perp,\ \|\mathbf z\|_\omega\leq 1,\\
    \mathbf y \in \mathbf D_\omega \cap \mathbf M_\omega^\perp,\ \|\mathbf y\|_\omega\leq 1
  \end{array}
  \right\}
  = \alpha,
\end{align}
where in the first two inequalities, we take the supremum over a larger set. In the fourth line the equality holds since for every $\mathbf x = \{x_i\}_{i=0}^r \in \mathbf H_\omega$ such that $x_i\in M_i$ and $\sum_{i=1}^r \omega_i \|P_{M_i\cap M^\perp}(x_i)\|^2\leq 1$, there is at least one $\mathbf z = \{z_i\}_{i=0}^r \in \mathbf H_\omega$ with $z_i\in M_i\cap M^\perp$ and $\|\mathbf z\|_\omega^2\leq 1$ for which the equality $\sum_{i=1}^r \omega_i \langle x_i, y\rangle = \sum_{i=1}^r \omega_i \langle z_i, y\rangle$ holds for all $y\in M^\perp$. For example, by \eqref{pr:cosCD1:reduction}, one can take $z_i:=P_{M_i\cap M^\perp}(x_i)$. This shows that $\cos_\omega(\mathbf C_\omega, \mathbf D_\omega) = \alpha$.
\end{proof}

\begin{lemma} \label{lem:cCD1}
Let $\omega \in \Omega_r$.  The configuration constant between $\mathbf C_\omega$ and $\mathbf D_\omega$ satisfies:
\begin{align} \nonumber \label{lem:cCD1:sup}
  c_\omega(\mathbf C_\omega, \mathbf D_\omega)
  & = \sup \left\{ \langle \mathbf x, \mathbf y \rangle_\omega \colon
  \begin{array}{l}
    \mathbf x \in \mathbf C_\omega \cap \mathbf M_\omega^\perp \cap \mathbf \Delta_\omega,\\
    \mathbf y \in \mathbf D_\omega \cap \mathbf M_\omega^\perp \cap \mathbf \Delta_\omega
  \end{array}
  \right\}\\
  & = \sup
  \left\{
    \sum_{i=1}^{r} \omega_i \langle x_i, y\rangle \colon
    \begin{array}{l}
      x_i \in M_i \cap M^\perp,\ \| x_i\| \leq 1,\ i=1,\ldots,r\\
      y \in M^\perp,\ \|y\| \leq 1
    \end{array}
  \right\}.
\end{align}
\end{lemma}

\begin{proof}
  The argument is similar to the one presented in the proof of Lemma \ref{lem:cosCD1}, where one should write ``$\sup_{i=1,\ldots,r}\|\cdot\|\leq 1$'' instead of ``$\sum_{i=1}^{r}\omega_i\|\cdot\|^2\leq 1$'' in \eqref{pr:cosCD1:cosCD}. We leave the details to the reader.
\end{proof}

\begin{lemma} \label{lem:cosCD2}
Let $\omega \in \Omega_r$. The following estimates hold:
\begin{equation}\label{lem:cosCD2:ineq}
  \cos_\omega(\mathbf C_\omega, \mathbf D_\omega)^2
  \leq c_\omega(\mathbf C_\omega, \mathbf D_\omega)
  \leq \cos_\omega(\mathbf C_\omega, \mathbf D_\omega).
\end{equation}
\end{lemma}

\begin{proof}

  By Theorem \ref{thm:norm}, we obtain
  \begin{equation}\label{}
    \cos_\omega(\mathbf C_\omega, \mathbf D_\omega)^2
    = \|T_\omega - P_M\|
    = \sup \left\{\left\|  \sum_{i=1}^{r}  \omega_i(P_{M_i}(x) - P_M(x)) \right\|
      \colon x\in M^\perp,\ \|x\|\leq 1\right\}.
  \end{equation}
  Note that, by \eqref{pr:cosCD1:PM_PMi:perp}, for all $x\in M^\perp$, we have
  \begin{equation}
    P_{M_i}(x) - P_M(x) = P_{M_i}(x) = P_{M_i}P_{M^\perp}(x) = P_{M_i\cap M^\perp}(x) \in M_i\cap M^\perp
  \end{equation}
  and $\|P_{M_i\cap M^\perp}(x)\| \leq \|x\| \leq 1$. Consequently,
  \begin{equation}\label{pr:cosCD2}
    \cos_\omega(\mathbf C_\omega, \mathbf D_\omega)^2
    \leq \sup \left\{\left\|\sum_{i=1}^{r}\omega_i x_i \right\|
    \colon x_i\in M_i \cap M^\perp,\ i=1,\ldots,r,\ \|x_i\|\leq 1\right\}.
  \end{equation}
  By the Riesz representation theorem, \eqref{def:Aweak}, the assumption that $x_i\in M_i \cap M^\perp$ and the fact that $\sum_{i=1}^r \omega_i x_i \in M^\perp$, we get
  \begin{equation}\label{}
    \left\|\sum_{i=1}^{r}\omega_i x_i \right\|
    = \sup \left\{ \sum_{i=1}^{r}\omega_i \langle x_i, y\rangle \colon y\in M^\perp, \|y\|\leq 1 \right\}
    \leq c_\omega(\mathbf C_\omega, \mathbf D_\omega).
  \end{equation}
  This, when combined with \eqref{pr:cosCD2}, proves the first inequality in \eqref{lem:cosCD2:ineq}. The second inequality in \eqref{lem:cosCD2:ineq} is trivial.
\end{proof}

\begin{remark}\label{rem:cosCD3}
  Let $\omega \in \Omega_r$ and assume that $J:=\{j \colon M_j \neq M\} \neq \emptyset$. Note that $M_j\cap M^\perp \neq \{0\}$ for all $j \in J$ and $M_j\cap M^\perp = \{0\}$ whenever $j\notin J$; compare with the proof of Proposition \ref{prop:inclusion}. It is not difficult to show that
  \begin{align}\label{rem:cosCD3:a1}
    \cos_{\omega}(\mathbf C_\omega, \mathbf D_\omega)
    & = \sup \left\{
      \sum_{j\in J} \omega_j \langle x_j, y\rangle \colon
      \begin{array}{l}
        x_j \in M_j \cap M^\perp,\  j \in J,\  \sum_{j\in J}\omega_j \| x_j\|^2 = 1,\\
        y \in M^\perp,\ \|y\| = 1
      \end{array}
    \right\} \\ \label{rem:cosCD3:a2}
    & = \sup \left\{ \left\|\sum_{j\in J}\omega_j x_j \right\|
      \colon x_j \in  M_j  \cap M^\perp,\   j \in J,\  \sum_{j\in J}\omega_j \|x_j\|^2 = 1\right\}\\ \label{rem:cosCD3:a3}
    &  = \sup \left\{
      \frac{\| \sum_{j\in J}  \omega_j x_j\|}{\sqrt{  \sum_{j\in J}  \omega_j \|x_j\|^{2}}}
      \colon x_j \in M_j\cap M^\perp,\  j \in J,\   0 \neq \sum_{j \in J} \omega_j \|x_j\|^2 < \infty \right\}
  \end{align}
  and
  \begin{align} \label{rem:cosCD3:b1}
    c_{\omega}(\mathbf C_\omega, \mathbf D_\omega)
    & = \sup \left\{
      \sum_{j\in J} \omega_j \langle x_j, y\rangle \colon
      \begin{array}{l}
        x_j \in M_j \cap M^\perp,\ \| x_j\| = 1,\  j \in J,  \\
        y \in M^\perp,\ \|y\| = 1
      \end{array}
    \right\} \\ \label{rem:cosCD3:b2}
    & = \sup \left \{ \left\|\sum_{j \in J}\omega_j x_j \right\| \colon x_j \in M_j\cap M^\perp,\  \|x_j\| = 1,\  j \in J  \right\}.
  \end{align}
\end{remark}

\begin{theorem}[Reduction to Unique Subspaces]\label{th:reduction}
Let $q\in \mathbb Z_+ \cup\{\infty\}$ be such that $q \leq r$, let $L_j$ be nontrivial, closed and linear  subspaces of  $\mathcal H$, $j=1,\ldots,q$, and let $L$ be their intersection. Moreover, let $\{I_j\}_{j=1}^q$ consist of nonempty, pairwise disjoint subsets of $\{1,\ldots,r\}$, possibly infinite, such that $\bigcup_{j=1}^q I_j = \{1,\ldots,r\}$, and assume that $M_i = L_j$ for all $i\in I_j$. Then,
\begin{equation}\label{th:reduction:eq}
  \cos_{\omega}(\mathbf C_\omega, \mathbf D_\omega) = \cos_{\lambda}(\mathbf E_\lambda, \mathbf D_\lambda)
  \quad \text{and} \quad
  c_{\omega}(\mathbf C_\omega, \mathbf D_\omega) = c_{\lambda}(\mathbf E_\lambda, \mathbf D_\lambda),
\end{equation}
where $\lambda := \{\lambda_j\}_{j=1}^q \in \Omega_q$ with $\lambda_j:= \sum_{i\in I_j} \omega_i$ and where (compare with \eqref{def:CinHw})
\begin{equation}\label{}
  \mathbf E_{\lambda} :=
  \begin{cases}
    \prod_{j=1}^{q}L_j, & \text{if } q \in \mathbb Z_+  \\
    \{\{u_j\}_{j=1}^{\infty} \colon u_j \in L_j,\ j=1,2, \ldots,\ \sum_{j=1}^{\infty}\lambda_j\| u_j\|^2 < \infty\}, & \text{if } q = \infty.
  \end{cases}
\end{equation}
\end{theorem}

\begin{proof}
We demonstrate only the first equality in \eqref{th:reduction:eq} by using Lemma \ref{lem:cosCD1}. A similar argument, when combined with Lemma \ref{lem:cCD1}, can be used to establish the second equality in \eqref{th:reduction:eq}.

In order to show the inequality $\cos_{\omega}(\mathbf C_\omega, \mathbf D_\omega) \leq \cos_{\lambda}(\mathbf E_\lambda, \mathbf D_\lambda)$, for each pair of points
\begin{equation}\label{pr:reduction:xy}
  \mathbf x = \{x_i\}_{i=1}^r \in \mathbf C_\omega \cap \mathbf M_\omega^\perp \cap \mathbf B_\omega
  \quad \text{and} \quad
  \mathbf y = \{y\}_{i=1}^r \in \mathbf D_\omega \cap \mathbf M_\omega^\perp \cap \mathbf B_\omega
\end{equation}
in $\mathbf H_\omega$ we define another pair
\begin{equation}\label{pr:reduction:uv}
  \mathbf u = \{u_j\}_{j=1}^q \in \mathbf E_\lambda \cap \mathbf L_\lambda^\perp \cap \mathbf B_\lambda
  \quad \text{and} \quad
  \mathbf v = \{v\}_{j=1}^q \in \mathbf D_\lambda \cap \mathbf L_\lambda^\perp \cap \mathbf B_\lambda
\end{equation}
in $\mathbf H_\lambda$ which satisfies
\begin{equation}\label{pr:reduction:xyuv}
  \langle \mathbf x, \mathbf y \rangle_\omega = \langle \mathbf u, \mathbf v \rangle_\lambda,
\end{equation}
where analogously to \eqref{def:MinHw},
\begin{equation}\label{}
  \mathbf L_{\lambda} :=
  \begin{cases}
    \prod_{j=1}^{q}L, & \text{if } q \in \mathbb Z_+  \\
    \{\{u_j\}_{j=1}^{\infty} \colon u_j \in L,\ j=1,2, \ldots,\ \sum_{j=1}^{\infty}\lambda_j\| u_j\|^2 < \infty\}, & \text{if } q = \infty.
  \end{cases}
\end{equation}
To this end, for each $j\in \{1,\ldots,q\}$, define $u_j:=\frac 1 {\lambda_j}\sum_{i\in I_j}\omega_i x_i$. Notice that $u_j$ is well defined and $u_j \in L_j$. Moreover, we have $\sum_{i=1}^{r}\omega_i x_i = \sum_{j=1}^{q}\lambda_j u_j$, where for $r = \infty$ we use Lemma \ref{lem:rearrangement}. By the convexity of $\|\cdot\|^2$,
\begin{equation}\label{}
  \sum_{j=1}^{q}\lambda_j \|u_j\|^2
  = \sum_{j=1}^{q}\lambda_j \left\|\sum_{ i\in I_j } \frac{\omega_i}{\lambda_j}x_i \right\|^2
  \leq \sum_{j=1}^{q}\lambda_j \sum_{i\in I_j} \frac{\omega_i}{\lambda_j} \|x_i\|^2
  = \|\mathbf x\|_\omega^2 <\infty,
\end{equation}
that is, $\mathbf u \in \mathbf H_\lambda$. On the other hand, since $L=M$, we can define $v:=y$. It is not difficult to see that with the above defined $\mathbf u$ and $\mathbf v$, equality \eqref{pr:reduction:xyuv} holds.

In order to prove the opposite inequality $\cos_{\omega}(\mathbf C_\omega, \mathbf D_\omega) \geq \cos_{\lambda}(\mathbf E_\lambda, \mathbf D_\lambda)$, this time for each pair in \eqref{pr:reduction:uv} we define the corresponding pair in \eqref{pr:reduction:xy}, for which again equality \eqref{pr:reduction:xyuv} holds. It suffices to take $x_i:=u_j$ for all $i \in I_j$ and $y:=v$. Indeed, by assumption, $x_i\in M_i$. Moreover, $\sum_{i\in I_j}\omega_i x_i = \lambda_j u_j$, hence $\sum_{i=1}^{r}\omega_i x_i = \sum_{j=1}^{q}\lambda_j u_j$. Furthermore,
\begin{equation}\label{}
  \sum_{i=1}^{r} \omega_i \|x_i\|^2 = \sum_{j=1}^{q} \lambda_j \|u_j\|^2 = \|\mathbf u\|_\lambda^2 <\infty,
\end{equation}
which shows that $\mathbf x \in \mathbf H_\omega$. Clearly, equality \eqref{pr:reduction:xyuv} holds for the pair $(\mathbf x, \mathbf y)$ defined above. This completes the proof.
\end{proof}

\begin{theorem}[Approximation, $r=\infty$] \label{th:approximation}
  Let $\omega \in \Omega_\infty$. Then
  \begin{equation}\label{th:approximation:eq}
    \cos_{\omega}(\mathbf C_\omega, \mathbf D_\omega)
    = \lim_{q\to\infty} \cos_{\lambda_q}(\mathbf C_{\lambda_q}, \mathbf D_{\lambda_q})
    \quad\text{and}\quad
    c_{\omega}(\mathbf C_\omega, \mathbf D_\omega)
    = \lim_{q\to\infty} c_{\lambda_q}(\mathbf C_{\lambda_q}, \mathbf D_{\lambda_q}),
  \end{equation}
  where $\lambda_q := \left\{\frac{\omega_1}{s_q},\ldots,\frac{\omega_q}{s_q} \right\} \in \Omega_q$ with $s_q := \sum_{i=1}^{q} \omega_i$  and $q = 2,3,\ldots$.
\end{theorem}

\begin{proof}
  We only show the first equality in \eqref{th:approximation:eq}. A similar argument can be employed to prove the second equality.

  Note that for all $q=2,3,\ldots,$ we have $s_q < s_{q+1}<1$, $N_q^\perp \subset N_{q+1}^\perp \subset M^\perp,$ where $N_q:=\bigcap_{i=1}^q M_i$ and
  \begin{equation}\label{}
    \mathbf C_{\lambda_q} \cap \mathbf M_{\lambda_q}^\perp \cap \mathbf B_{\lambda_q}
    = \left\{
    \{x_i\}_{i=1}^{q} \colon
      \begin{array}{l}
        x_i \in M_i \cap N_q^\perp,\ i=1,\ldots,q,\\
        \sum_{i=1}^{q}\omega_i \|x_i\|^2 \leq s_{q}
      \end{array}
    \right\}.
  \end{equation}
  Consequently, if $\{x_1,\ldots,x_q\} \in \mathbf C_{\lambda_q} \cap \mathbf M_{\lambda_q}^\perp \cap \mathbf B_{\lambda_q}$, then  $\{x_1,\ldots,x_q, 0\} \in \mathbf C_{\lambda_{q+1}} \cap \mathbf M_{\lambda_{q+1}}^\perp \cap \mathbf B_{\lambda_{q+1}}$ and analogously $\{x_1,\ldots,x_q, 0, 0, \ldots\} \in \mathbf C_\omega \cap \mathbf M_\omega^\perp \cap \mathbf B_\omega$. Hence, by Lemma \ref{lem:cosCD1},
  \begin{equation}\label{}
    s_q \cos_{\lambda_q}(\mathbf C_{\lambda_q}, \mathbf D_{\lambda_q})
    \leq s_{q+1} \cos_{\lambda_{q+1}}(\mathbf C_{\lambda_{q+1}}, \mathbf D_{\lambda_{q+1}})
    \leq \cos_{\omega}(\mathbf C_\omega, \mathbf D_\omega)
  \end{equation}
  and thus the sequence $\{s_q \cos_{\lambda_q}(\mathbf C_{\lambda_q}, \mathbf D_{\lambda_q})\}_{q=2}^\infty$ is monotone and bounded, and therefore converges to some number $\alpha$. Moreover, $\alpha \leq \cos_{\omega}(\mathbf C_\omega, \mathbf D_\omega)$.

  In order to show the opposite inequality $\alpha \geq \cos_{\omega}(\mathbf C_\omega, \mathbf D_\omega)$, we first demonstrate that for each pair
  \begin{equation}\label{pr:approximation:xy}
    \mathbf x = \{x_i\}_{i=1}^\infty \in \mathbf C_\omega \cap \mathbf M_\omega^\perp \cap \mathbf B_\omega
    \quad\text{and}\quad
    \mathbf y = \{y\}_{i=1}^\infty \in \mathbf C_\omega \cap \mathbf M_\omega^\perp \cap \mathbf B_\omega
  \end{equation}
  in $\mathbf H_\omega$ and for each $\varepsilon >0$, we can find another pair
  \begin{equation}\label{pr:approximation:xqyq}
    \mathbf x_q = \{x_{q,i}\}_{i=1}^q \in \mathbf C_{\lambda_q} \cap \mathbf M_{\lambda_q}^\perp \cap \mathbf B_{\lambda_q}
    \quad\text{and}\quad
    \mathbf y_q = \{y_q\}_{i=1}^q \in \mathbf C_{\lambda_q} \cap \mathbf M_{\lambda_q}^\perp \cap \mathbf B_{\lambda_q}
  \end{equation}
  in $\mathbf H_{\lambda_q}$ such that
  \begin{equation}\label{pr:approximation:xyANDxqyq}
    \langle \mathbf x, \mathbf y\rangle_\omega
    \leq s_q \langle \mathbf x_q, \mathbf y_q\rangle_{\lambda_q} + \varepsilon.
  \end{equation}
  To this end, choose $n$ so that the tail satisfies $\sum_{i=n+1}^{\infty} \omega_i \langle x_i, y\rangle < \frac \varepsilon 2$. For each $q > n$, define
  \begin{equation}\label{}
    x_{q,i}:=
    \begin{cases}
      \sqrt{s_q} \cdot P_{M_i \cap N_q^\perp}(x_i), & \mbox{if } i=1,\ldots,n \\
      0, & \mbox{otherwise}
    \end{cases}
    \quad\text{and}\quad
    y_q := \sqrt{s_q} \cdot P_{N_q^\perp}(y).
  \end{equation}
  Note that $\mathbf x_q$ and $\mathbf y_q$ indeed satisfy \eqref{pr:approximation:xqyq} as $\|\mathbf x_q\|_{\lambda_q}\leq 1$ and $\|\mathbf y_q\| \leq 1$.

  On the other hand, the decreasing sequence of sets $\{N_q\}_{q=2}^\infty$ converges to $M$ in the sense of Mosco; see \cite[Definition 1.2 and Lemma 3.1]{Mosco1969}. This, when combined with \cite[Theorem 3.2]{Tsukada1984}, implies that for all $x \in \mathcal H$, we get $P_{N_q}(x) \to P_M(x)$ as $q\to \infty$ and further, $    P_{N_q^\perp} (x) = x - P_{N_q}(x) \to x - P_M(x) = P_{M^\perp}(x)$ as $q\to \infty$. In this connection, see also \cite[Proposition 7]{IsraelMaximianoReich1983} and \cite[Lemma 4.2]{BauschkeBorwein1996}.

  In particular, using the assumptions that $x_i \in M_i\cap M^\perp$, $y\in M^\perp$ and the equality $P_{M_i\cap N_q^\perp} = P_{N_q^\perp}P_{M_i}$ (compare with \eqref{pr:cosCD1:PM_PMi:perp}), we obtain
  \begin{equation}\label{}
     x_{q,i} \to P_{M^\perp}(x_i) = x_i
    \quad\text{and}\quad
    y_q \to P_{M^\perp}(y) = y
  \end{equation}
  as $q \to \infty$. Consequently, for all large enough $q$ and for all $i=1,\ldots,n$, we reach the inequality $\langle x_i, y\rangle \leq \langle x_{q,i}, y_q\rangle + \frac{\varepsilon}{2n}$, which leads to
  \begin{equation}\label{}
    \langle \mathbf x, \mathbf y\rangle_\omega
    \leq \sum_{i=1}^{n}\omega_i\langle x_i, y \rangle + \frac \varepsilon 2
    \leq \sum_{i=1}^{n}\omega_i\langle x_{q,i}, y_q \rangle + \varepsilon
    = s_q \langle \mathbf x_q, \mathbf y_q\rangle_{\lambda_q} + \varepsilon.
  \end{equation}
  This shows \eqref{pr:approximation:xyANDxqyq}, as claimed.

  We are now ready to return to the inequality $\alpha \geq \cos_{\omega}(\mathbf C_\omega, \mathbf D_\omega)$. Indeed, by \eqref{pr:approximation:xyANDxqyq}, and by the monotonicity of the sequence $\{s_q \cos_{\lambda_q}(\mathbf C_{\lambda_q}, \mathbf D_{\lambda_q})\}_{q=2}^\infty$, we have
  \begin{equation}\label{}
    \langle \mathbf x, \mathbf y\rangle_\omega
    \leq s_q \langle \mathbf x_q, \mathbf y_q\rangle_{\lambda_q} + \varepsilon
    \leq s_q \cos_{\lambda_q}(\mathbf C_{\lambda_q}, \mathbf D_{\lambda_q}) + \varepsilon
    \leq \alpha + \varepsilon.
  \end{equation}
  By taking the supremum over all $\mathbf x$ and $\mathbf y$ satisfying \eqref{pr:approximation:xy}, we obtain that $\cos_{\omega}(\mathbf C_\omega, \mathbf D_\omega) \leq \alpha + \varepsilon$, which proves the asserted inequality and hence completes the proof.
\end{proof}

\begin{proposition}\label{prop:orthogonal}
  Let $\omega \in \Omega_r$ and assume that the subspaces $M_1,\ldots,M_r$ are nontrivial and pairwise orthogonal. Then
  \begin{equation}\label{prop:orthogonal:eq}
    c_{\omega}(\mathbf C_\omega, \mathbf D_\omega)
    = \sqrt{\sum_{i=1}^{r}\omega_i^2}
    \qquad \text{and} \qquad
    \cos_{\omega}(\mathbf C_\omega, \mathbf D_\omega)
    = \sqrt{\sup_{i=1,\ldots,r}\omega_i}.
  \end{equation}
\end{proposition}

\begin{proof}
  Observe  first that we must have $J:=\{j\colon M_j \neq M\} = \{1,\ldots,r\}$. Otherwise, there would be a pair $i$ and $j$ for which $M_i = M \subset M_j$ and $M_i \perp M_j$, which is possible only when $M_i=\{0\}$, a contradiction. Moreover, $\sup_i \omega_i = \omega_j$ for some $j \in \{1,\ldots,r\}$, even when $r = \infty$. By the assumed pairwise orthogonality, for all $\{x_i\}_{i=1}^r$ with $x_i\in M_i\cap M^\perp$ and $\sum_{i=1}^{r}\omega_i \|x_i\|^2 = 1$, we have
  \begin{equation}\label{pr:orthogonal:eq}
    \left\|\sum_{i=1}^{r}\omega_i x_i \right\|^2
    = \sum_{i=1}^{r}\omega_i^2 \|x_i\|^2
    \leq \omega_j \sum_{i=1}^{r}\omega_i \|x_i\|^2 = \omega_j.
  \end{equation}
  Thus, by \eqref{rem:cosCD3:a2}, $\cos_{\omega}(\mathbf C_\omega, \mathbf D_\omega) \leq \omega_j$. On the other hand, when  $\|x_j\| = \frac 1 {\sqrt{\omega_j}}$,  the assumption $\sum_{i=1}^{\infty}\omega_i \|x_i\|^2 = 1$ implies that $\|x_i\| = 0$ for all $i\neq j$. Hence, in this case, $\|\sum_{i=1}^{r}\omega_i x_i\|^2 = \omega_j$ and therefore $\cos_{\omega}(\mathbf C_\omega, \mathbf D_\omega) = \omega_j$. In view of \eqref{rem:cosCD3:b2} and \eqref{pr:orthogonal:eq}, it is not difficult to see that $c_{\omega}(\mathbf C_\omega, \mathbf D_\omega) = \sqrt{\sum_{i=1}^{r}\omega_i^2}$.
\end{proof}

In the next example, we show that the equality or inequality between $c_{\omega}(\mathbf C_\omega, \mathbf D_\omega)$ and $\cos_{\omega}(\mathbf C_\omega, \mathbf D_\omega)$ may depend on the weights $\omega \in \Omega_r$.

\begin{example}\label{ex:orthogonal}
Let $q\in \mathbb Z_+$ be such that $q\leq r$ and let $L_1,\ldots,L_q$ be a tuple of nontrivial, closed and linear subspaces of $\mathcal H$, which are pairwise orthogonal. Assume that the list $M_1, \ldots, M_r $ consists only of subspaces from $L_1,\ldots, L_q$ and let $\lambda \in \Omega_q$ be defined as in Theorem \ref{th:reduction}. Then, by Theorem \ref{th:reduction} and Proposition \ref{prop:orthogonal}, we have $c_{\omega}(\mathbf C_\omega, \mathbf D_\omega)
= \sqrt{\sum_{j=1}^{q}\lambda_j^2}$ and $\cos_{\omega}(\mathbf C_\omega, \mathbf D_\omega) = \sqrt{\sup_{j=1,\ldots,q}\lambda_j}$. We consider two cases:
\begin{enumerate}[(a)]
  \item If $\lambda_j = 1/q$ for all $j=\{1,\ldots,q\}$, then $c_{\omega}(\mathbf C_\omega, \mathbf D_\omega) = \cos_{\omega}(\mathbf C_\omega, \mathbf D_\omega) = \sqrt{1/q}$.
  \item If $\lambda_i:=\max_{j=1,\ldots,q}\lambda_j > 1/q$, then $\sum_{j=1}^{q}\lambda_j^2 < \lambda_i \sum_{j=1}^{q}\lambda_j = \lambda_i$ and consequently, $c_{\omega}(\mathbf C_\omega, \mathbf D_\omega) < \cos_{\omega}(\mathbf C_\omega, \mathbf D_\omega)$.
\end{enumerate}
\end{example}

In spite of the previous example, the ``parallel'' alignment of the subspaces $M_i$ does not depend on $\omega$, as we show in our next theorem.

\begin{theorem} \label{thm:cosCDequal1}
The following conditions are equivalent:
\begin{enumerate}[(i)]
  \item There is $\lambda = \{\lambda_i\}_{i=1}^r \in \Omega_r$ such that $\cos_\lambda(\mathbf C_\lambda, \mathbf D_\lambda) = 1$.
  \item For all $\omega = \{\omega_i\}_{i=1}^r \in \Omega_r$, we have $\cos_\omega(\mathbf C_\omega, \mathbf D_\omega) = 1$.
\end{enumerate}
\end{theorem}

\begin{proof}
  Assume that $\cos_\lambda(\mathbf C_\lambda, \mathbf D_\lambda) = 1$. It suffices to show that there is at least one sequence of pairs $\{\mathbf x_k, \mathbf y_k\}_{k=1}^\infty$, which for all $\omega \in \Omega_r$ satisfies
  \begin{equation}\label{pr:cosCDequal1:toShowA}
    \mathbf x_k = \{x_{k,i}\}_{i=1}^r \in \mathbf C_\omega \cap \mathbf M_\omega^\perp \cap \mathbf B_\omega, \qquad
    \mathbf y_k = \{y_k\}_{i=1}^r \in \mathbf D_\omega \cap \mathbf M_\omega^\perp \cap \mathbf B_\omega
  \end{equation}
  and
  \begin{equation}\label{pr:cosCDequal1:toShowB}
    \lim_{k\to \infty} \langle \mathbf x_k, \mathbf y_k \rangle_\omega =1.
  \end{equation}

  By Lemma \ref{lem:cosCD2}, $c_\lambda(\mathbf C_\lambda, \mathbf D_\lambda) = 1$. Observe that we must have $J:=\{j\colon M_j \neq M\} = \{1,\ldots,r\}$. Otherwise, by \eqref{rem:cosCD3:b1} and by the Cauchy-Schwarz inequality (applied to each summand), we would arrive at the following contradiction:
  $ 1 = c_\lambda(\mathbf C_\lambda, \mathbf D_\lambda) \leq \sum_{j\in J} \lambda_j <1 $. Consequently, by the definition of the supremum in \eqref{rem:cosCD3:b1}, for each $k=1,2,\ldots$, there are
  \begin{equation}\label{pr:cosCDequal1:xk}
    \mathbf x_k = \{x_{k,i}\}_{i=1}^r, \qquad x_{k,i}\in M_i\cap M^\perp, \qquad \|x_{k,i}\| = 1
  \end{equation}
  and
  \begin{equation}\label{pr:cosCDequal1:yk}
    \mathbf y_k = \{y_k\}_{i=1}^r, \qquad y_k\in M^\perp, \qquad \|y_k\| = 1,
  \end{equation}
  which satisfy
  \begin{equation}\label{pr:cosCDequal1:supLambda}
    1-\frac 1 k \leq \langle \mathbf x_k, \mathbf y_k \rangle_\lambda
    = \sum_{i=1}^{r}\lambda_i \langle x_{k,i}, y_k \rangle \leq 1.
  \end{equation}
  Without any loss to the generality we may assume that $\langle x_{k,i}, y_k\rangle \geq 0$ for all $i=1,\ldots,r$ and all $k=1,2,\ldots$. Indeed, if $\langle x_{k,i}, y_k\rangle < 0$ for some $k$ and $i$ then, by replacing ``$x_{k,i}$'' by ``$-x_{k,i}$'', we can only increase the number $\langle \mathbf x_k, \mathbf y_k \rangle_\lambda$ in \eqref{pr:cosCDequal1:supLambda}.

  Obviously, the above-defined sequence of pairs $\{\mathbf x_k, \mathbf y_k\}_{k=1}^\infty$ satisfies \eqref{pr:cosCDequal1:toShowA} for all $\omega \in \Omega_r$ and \eqref{pr:cosCDequal1:toShowB} for $\omega = \lambda$. What remains to be shown is that $\{\mathbf x_k, \mathbf y_k\}_{k=1}^\infty$ satisfies \eqref{pr:cosCDequal1:toShowB} for all $\omega \in \Omega_r$. Before doing so, we investigate the properties of $\langle x_{k,i}, y_k \rangle$ in more detail.

  Note that by the Cauchy-Schwarz inequality, we have $\langle x_{k,i}, y_k\rangle \leq 1$. We now show that for each $i=1,\ldots,r$, we have
  \begin{equation}\label{pr:cosCDequal1:toShowC}
    \lim_{k\to\infty}\langle x_{k,i}, y_k \rangle = 1.
  \end{equation}
  Suppose to the contrary that
  \begin{equation}
    \liminf_{k\to\infty}\langle x_{k,j}, y_k \rangle
    = \lim_{n\to\infty} \langle x_{k_n,j}, y_{k_n} \rangle
    = 1 - \varepsilon
  \end{equation}
  for some $j$ and some $\varepsilon \in (0,1]$, where $\{k_n\}_{n=1}^\infty$ is a subsequence of $\{k\}_{k=1}^\infty$. By taking $n$ large enough, we may assume that $\langle x_{k_n,j}, y_{k_n} \rangle \leq 1 - \frac \varepsilon 2$. This, when combined with \eqref{pr:cosCDequal1:supLambda}, leads to
  \begin{equation}\label{}
     1-\frac 1 {k_n}
    \leq \lambda_j \langle x_{k_n,j}, y_{k_n} \rangle
      + \sum_{i\neq j} \lambda_i \langle x_{k_n,i}, y_{k_n} \rangle
    \leq \lambda_j (1 - \frac \varepsilon 2) + \sum_{i\neq j} \lambda_i
    =  1-\lambda_j \frac \varepsilon 2 < 1,
  \end{equation}
  which is a contradiction since the left-hand side converges to one as $n\to\infty$.

  We are now ready to show that the above-defined sequence of pairs $\{\mathbf x_k, \mathbf y_k\}_{k=1}^\infty$ satisfies \eqref{pr:cosCDequal1:toShowB} for all $\omega \in \Omega_r$. Indeed, let $\omega \in \Omega_r$. Moreover, let $\varepsilon \in (0,1)$ and let $n \in \{1,\ldots,r\}$ be an integer such that $\sum_{i=1}^{n}\omega_i \geq \sqrt{1-\varepsilon}$. Obviously, when $r\in \mathbb Z_+$, we can take $n:=r$.  By \eqref{pr:cosCDequal1:toShowC},  we may assume that $\langle x_{k,i}, y_k \rangle \geq \sqrt{1-\varepsilon}$ for all $i=1,2,\ldots,n$ and all large enough $k \geq K_n$. Thus, for all $k\geq K_n$, we arrive at
  \begin{equation}\label{}
    1 \geq \langle \mathbf x_k, \mathbf y_k \rangle_\omega
    \geq \sum_{i=1}^{n}\omega_i \langle x_{k,i}, y_k \rangle
    \geq \sum_{i=1}^{n}\omega_i \sqrt{1-\varepsilon}
    \geq 1-\varepsilon,
  \end{equation}
  which shows that $\langle \mathbf x_k, \mathbf y_k \rangle_\omega \to 1$ as $k\to \infty$. This proves \eqref{pr:cosCDequal1:toShowB} and completes the proof of the lemma itself.
\end{proof}

\begin{remark}[Erratum to \cite{ReichZalas2017}]\label{rem:erratum}
  As we have already observed in Proposition \ref{prop:inclusion}, $\mathbf C_{\omega} \cap \mathbf M_{\omega}^\perp$ may be a proper subset of $ \mathbf C_{\omega} \cap (\mathbf C_{\omega} \cap \mathbf D_{\omega})^\perp$ and equality need not hold. Consequently, the argument used in the proof of \cite[Theorem 8]{ReichZalas2017} preceding \cite[equality (19)]{ReichZalas2017} was incorrect. However, Lemma \ref{lem:cosCD1} justifies the validity of \cite[equality (19)]{ReichZalas2017} because
  \begin{equation}\label{}
    \cos_\omega(\mathbf C_\omega, \mathbf D_\omega)
    = \sup \left\{ \frac{\langle \mathbf x, \mathbf y \rangle_\omega}{\|\mathbf x\|_\omega \|\mathbf y\|_\omega}
    \colon
    \begin{array}{l}
      \mathbf x \in \mathbf C_\omega \cap \mathbf M_\omega^\perp,\ \mathbf x \neq \mathbf 0\\
      \mathbf y \in \mathbf D_\omega \cap \mathbf M_\omega^\perp,\ \mathbf y \neq \mathbf 0
    \end{array}
    \right\}.
  \end{equation}
\end{remark}

\section{Asymptotic Properties of the Simultaneous Projection Method} \label{sec:AsymptoticProp}

In this section we oftentimes refer to the subspace $A_\omega(\mathbf C_\omega^\perp)$, the explicit form of which is given in Proposition \ref{prop:Minkowski}. We begin with a theorem which corresponds to equivalence \eqref{int:equivalence}.

\begin{theorem} \label{thm:equiv}
  Let $\omega \in \Omega_r$. The following conditions are equivalent:
\begin{multicols}{2}
\begin{enumerate}[(i)]
  \item $\|T_\omega - P_{M}\|<1$;
  \item $\cos_{\omega}(\mathbf C_{\omega}, \mathbf D_{\omega})<1$;
  \item The set $A_\omega(\mathbf C_\omega^\perp)$ is closed in $\mathcal H$;

  \item $\|P_{\mathbf D_{\omega}} P_{\mathbf C_{\omega}} - P_{\mathbf C_{\omega} \cap \mathbf D_{\omega}}\|_{\omega} <1 $;
  \item $\{\mathbf C_{\omega}, \mathbf D_{\omega}\}$ is linearly regular;
  \item $\mathbf C_{\omega}^\perp + \mathbf D_{\omega}^\perp$ is closed in $\mathbf H_\omega$.
\end{enumerate}
\end{multicols}
\end{theorem}

\begin{proof}
By \eqref{int:equivalence} applied to $\mathbf C_\omega$ and $\mathbf D_\omega$, we have the equivalence between (ii) and (vi). Similarly, by \cite[Theorem 5.19]{BauschkeBorwein1996} applied to $\mathbf C_\omega$ and $\mathbf D_\omega$, we obtain the equivalence between (v) and (vi). Recall that $\{\mathbf C_{\omega}, \mathbf D_{\omega}\}$ is said to be linearly regular if the inequality $\max\{d(\mathbf x,\mathbf C_{\omega}), d(\mathbf x, \mathbf D_{\omega})\} \leq \kappa d(\mathbf x, \mathbf C_\omega \cap \mathbf D_\omega)$ holds for all $\mathbf x \in \mathbf H_\omega$ and some $\kappa >0$. Proposition \ref{prop:Minkowski} verifies the equivalence between (iii) and (vi). Finally, by  Theorems  \ref{thm:norm}  and \ref{int:th:APM},  we have
\begin{equation}\label{}
  \|T_\omega - P_M\| =  \cos_\omega (\mathbf C_\omega, \mathbf D_\omega)^2
  = \|P_{\mathbf D_{\omega}}P_{\mathbf C_{\omega}} - P_{\mathbf C_{\omega} \cap \mathbf D_{\omega}}\|_\omega,
\end{equation}
which explains the equivalence between (i), (ii) and (iv).
\end{proof}

\begin{theorem}[Dichotomy]\label{th:dichotomy}
  Exactly one of the following two statements holds:
  \begin{enumerate}[(i)]
    \item For all $\omega\in \Omega_r$, the  set $A_\omega(\mathbf C_\omega^\perp)$ is closed in $\mathcal H$.  Then  the sequence $\{T^k_{\omega}\}_{k=1}^\infty$  converges linearly to $P_M$ as $k\to \infty$ and the optimal error bound is given by
        \begin{equation}\label{th:dichotomy:estimate}
        \|T_\omega^k(x) - P_{M}(x)\| \leq \cos_{\omega}(\mathbf C_{\omega}, \mathbf D_{\omega})^{2k} \cdot \|x\|.
      \end{equation}
    \item For all $\omega \in \Omega_r$, the  set $A_\omega(\mathbf C_\omega^\perp)$ is not closed in $\mathcal H$.  Then  the sequence $\{T^k_{\omega}\}_{k=1}^\infty$  converges arbitrarily slowly to $P_{M}$ as $k\to \infty$.
  \end{enumerate}
\end{theorem}

\begin{proof}
  By combining Theorems \ref{thm:cosCDequal1} and \ref{thm:equiv}, we see that either  $A_\omega(\mathbf C_\omega^\perp)$ is closed for all $\omega \in \Omega_r$ or $A_\omega(\mathbf C_\omega^\perp)$ is not closed for all $\omega \in \Omega_r$.  This shows the dichotomy between (i) and (ii).

  If we assume as in (i)  that $A_\omega(\mathbf C_\omega^\perp)$ is closed,  then both the linear convergence and the optimality of the estimate \eqref{th:dichotomy:estimate} follow from Theorems \ref{thm:norm}  and \ref{thm:equiv}.

  Assume now that  $A_\omega(\mathbf C_\omega^\perp)$ is not closed,  where $\omega \in \Omega_r$. We show that the sequence $\{T_\omega^k\}_{k=1}^\infty$ converges arbitrarily slowly to $P_{M}$ as $k\to \infty$. To this end, let $\{a_k\}_{k=1}^\infty \subset [0,\infty)$ be a null sequence and let  $\{b_k\}_{k=1}^\infty$ be defined by $b_1 := a_1,\ b_2 := a_1$ and $b_k := a_{k-1},\ k \geq 3$.  By Theorem \ref{thm:equiv}, we see that $\mathbf C_{\omega}^\perp + \mathbf D_{\omega}^\perp$ is not closed in $\mathbf H_\omega$. This, when combined with Theorem \ref{int:th:dichotomy}, implies that the sequence $\{(P_{\mathbf D_{\omega}}P_{\mathbf C_{\omega}})^k\}_{k=1}^\infty$ converges to $P_{\mathbf C_{\omega} \cap \mathbf D_{\omega}}$ arbitrarily slowly as $k \to \infty$. In particular, there is $\mathbf x \in \mathbf H_\omega$, such that
  \begin{equation}\label{pr:dichotomy:b_k}
    \|(P_{\mathbf D_{\omega}}P_{\mathbf C_{\omega}})^k(\mathbf x) - P_{\mathbf C_{\omega} \cap \mathbf D_{\omega}} (\mathbf x)\| \geq b_k
  \end{equation}
  for all $k=1,2,\ldots$. Note that $\mathbf y := P_{\mathbf D_{\omega}}P_{\mathbf C_{\omega}} (\mathbf x) \in \mathbf D_\omega$, hence $\mathbf y = \{y\}_{i=1}^r$ for some $y \in \mathcal H$. Moreover, $P_{\mathbf C_{\omega} \cap \mathbf D_{\omega}} (\mathbf x) = P_{\mathbf C_{\omega} \cap \mathbf D_{\omega}} (\mathbf y)$ (compare with \eqref{pr:norm:PCD_PC_PD}). Consequently, by  rewriting \eqref{pr:dichotomy:b_k}  in terms of $\mathbf y$ and $a_k$, and by Theorem \ref{thm:normConvergence}, we arrive at
  \begin{equation}
    \|T_\omega^k(y)-P_M(y) \|
    = \|(P_{\mathbf D_{\omega}}P_{\mathbf C_{\omega}})^k(\mathbf y) - P_{\mathbf C_{\omega} \cap \mathbf D_{\omega}} (\mathbf y)\| \geq a_k
  \end{equation}
  for all $k=1,2,\ldots$. This shows that the sequence $\{T_\omega^k\}_{k=1}^\infty$ converges arbitrarily slowly to $P_{M}$ as $k\to \infty$, as asserted.
\end{proof}

\begin{theorem}[Super-polynomial Rate]\label{th:superPolyRate}
  Let $\omega \in \Omega_r$  and assume that the set $A_\omega(\mathbf C_\omega^\perp)$ is  not closed  in $\mathcal H$. Then  the sequence $\{T^k_{\omega}\}_{k=1}^\infty$ converges super-polynomially fast to $P_{M}$, as $k\to \infty$, on some dense linear subspace $Y_\omega \subset \mathcal H$.
\end{theorem}
\begin{proof}
  The argument follows the proof of \cite[Theorem 14]{ReichZalas2017}.  In view of Theorem \ref{thm:equiv}, the subspace $\mathbf C_\omega^\perp + \mathbf D_\omega^\perp$ is not closed.  By Theorem \ref{int:th:superPoly} applied to $\mathbf C_\omega$ and $\mathbf D_\omega$, there is a dense linear subspace $\mathbf X_\omega$ of $\mathbf H_\omega$ on which the sequence $\{(P_{\mathbf C_\omega}P_{\mathbf D_\omega})^k\}_{k=1}^\infty$ converges super-polynomially fast to $P_{\mathbf C_\omega \cap \mathbf D_\omega}$. Define
  \begin{equation}\label{}
    \mathbf Y_\omega:=P_{\mathbf D_\omega}(\mathbf X_\omega)
    \quad \text{and} \quad
    Y_\omega := \{y \in \mathcal H \colon \mathbf y = \{y\}_{i=1}^r \in \mathbf Y_\omega\}.
  \end{equation}
  Note that the linearity of $P_{\mathbf D_\omega}$ implies that $\mathbf Y_\omega$ and $Y_\omega$ are both linear subspaces of $\mathbf H_\omega$ and $\mathcal H$, respectively. Let $y \in Y_\omega$ and $\mathbf x \in \mathbf X_\omega$ be such that $\mathbf y = \{y\}_{i=1}^r = P_{\mathbf D_\omega}(\mathbf x)$.  Then, by Lemma \ref{lem:ProjCD}, \eqref{pr:norm:PCD_PC_PD}  and by the nonexpansivity of $P_{\mathbf D_\omega}$, for each $n=1,2,\ldots,$ we have
  \begin{align}\nonumber
    k^n \|T_\omega^k(y) - P_M(y)\|
    & = k^n \|(P_{\mathbf D_\omega}P_{\mathbf C_\omega})^k(\mathbf y) - P_{\mathbf C_\omega \cap \mathbf D_\omega}(\mathbf y) \|_\omega \\ \nonumber
    & = k^n \|P_{\mathbf D_\omega}(P_{\mathbf C_\omega}P_{\mathbf D_\omega})^k(\mathbf x) - P_{\mathbf D_\omega}P_{\mathbf C_\omega \cap \mathbf D_\omega}(\mathbf x) \|_\omega \\
    &  \leq k^n \|(P_{\mathbf C_\omega}P_{\mathbf D_\omega})^k(\mathbf x) - P_{\mathbf C_\omega \cap \mathbf D_\omega}(\mathbf x) \|_\omega \to 0
  \end{align}
  as $k \to \infty.$ This shows that the sequence $\{T^k_{\omega}\}_{k=0}^\infty$ converges super-polynomially fast to $P_{M}$ on $Y_\omega$.

  We now show that $Y_\omega$ is a dense linear subspace of $\mathcal H$. Indeed, let $x\in \mathcal H$ and let $\mathbf x := \{x\}_{i=1}^r \in \mathbf D_\omega$. Since $\mathbf X_\omega$ is dense in $\mathbf H_\omega$, there is $\{\mathbf x_k\}_{k=1}^\infty$ in $\mathbf X_\omega$ such that $\mathbf x_k \to \mathbf x$. Let $\mathbf y_k:= P_{\mathbf D_\omega}(\mathbf x_k)$. Since $\mathbf y_k \in \mathbf Y_\omega$, there is $y_k\in Y_\omega$ such that $\mathbf y_k= \{y_k\}_{i=1}^r$. Again, by Lemma \ref{lem:ProjCD} and by the nonexpansivity of $P_{\mathbf D_\omega}$, we arrive at
  \begin{equation}\label{}
    \|y_k -x\|
    = \|\mathbf y_k - \mathbf x\|_\omega
    = \|P_{\mathbf D_\omega}(\mathbf x_k) - P_{\mathbf D_\omega}(\mathbf x) \|_\omega
    \leq \|\mathbf x_k - \mathbf x\|_\omega \to 0
  \end{equation}
  as $k\to \infty$. This completes the proof.
\end{proof}

\begin{theorem}[Polynomial Rate]\label{th:polyRate}
  Let $\omega \in \Omega_r$. Assume that $y \in A_\omega(C_\omega^\perp)$. Then there is $C_\omega(y)>0$ such that for all $k$, we have
  \begin{equation}\label{th:polyRate:eq}
   \|T_\omega^k(y) - P_{M}(y)\| \leq \frac{C_\omega(y)}{\sqrt{k}}.
  \end{equation}
\end{theorem}
\begin{proof}
  We first show that in spite of the possible inequality $\mathbf C_\omega \cap \mathbf D_\omega \neq \{\mathbf 0\}$ (see Theorem \ref{int:th:poly}), for all $\mathbf x = \{x_i\}_{i=1}^r \in \mathbf C_\omega^\perp + \mathbf D_\omega^\perp$ there is $C_\omega(\mathbf x)>0$ such that
  \begin{equation}\label{pr:polyRate:step1}
    \|(P_{\mathbf D_\omega}P_{\mathbf C_\omega})^k(\mathbf x) - P_{\mathbf C_\omega \cap \mathbf D_\omega}(\mathbf x)\|_\omega
    \leq \frac{C_\omega(\mathbf x)}{\sqrt{k}}
  \end{equation}
  for all $k=1,2,\ldots$. To this end, let $\mathbf M_1 := \mathbf C_\omega \cap (\mathbf C_\omega \cap \mathbf D_\omega)^\perp$ and $\mathbf M_2 := \mathbf D_\omega \cap (\mathbf C_\omega \cap \mathbf D_\omega)^\perp$. Recall again that the projections $P_{\mathbf C_\omega}$ and $P_{\mathbf D_\omega}$ commute with $P_{\mathbf C_\omega \cap \mathbf D_\omega}$; see \eqref{pr:norm:PCD_PC_PD}.   Similarly to \eqref{pr:cosCD1:PM_PMi:perp}, they commute with $P_{(\mathbf C_\omega \cap \mathbf D_\omega)^\perp}$, that is,
  \begin{equation}\label{}
    P_{\mathbf M_1}
    = P_{(\mathbf C_\omega \cap \mathbf D_\omega)^\perp} P_{\mathbf C_\omega}
    = P_{\mathbf C_\omega} P_{(\mathbf C_\omega \cap \mathbf D_\omega)^\perp},
  \end{equation}
  \begin{equation}
    P_{\mathbf M_2}
    = P_{(\mathbf C_\omega \cap \mathbf D_\omega)^\perp} P_{\mathbf D_\omega}
    = P_{\mathbf D_\omega} P_{(\mathbf C_\omega \cap \mathbf D_\omega)^\perp}.
  \end{equation}
  Using their linearity and the above-mentioned commuting properties, we obtain
  \begin{align}\label{} \nonumber
    (P_{\mathbf D_\omega}P_{\mathbf C_\omega})^k - P_{\mathbf C_\omega \cap \mathbf D_\omega}
    & = (P_{\mathbf D_\omega}P_{\mathbf C_\omega})^k - (P_{\mathbf D_\omega}P_{\mathbf C_\omega})^k P_{\mathbf C_\omega \cap \mathbf D_\omega}\\ \nonumber
    & = (P_{\mathbf D_\omega}P_{\mathbf C_\omega})^k P_{(\mathbf C_\omega \cap \mathbf D_\omega)^\perp}\\ \nonumber
    & = (P_{\mathbf D_\omega}P_{\mathbf C_\omega})^k (P_{(\mathbf C_\omega \cap \mathbf D_\omega)^\perp})^{2k}\\
    & = (P_{\mathbf M_2} P_{\mathbf M_1})^k.
  \end{align}
  We may now apply Theorem \ref{int:th:poly} to $\mathbf M_1$ and $\mathbf M_2$ because $\mathbf M_1 \cap \mathbf M_2 = \{\mathbf 0\}$. Thus, for every $\mathbf x \in \mathbf M_1^\perp + \mathbf M_2^\perp$ there is $C_\omega(\mathbf x)>0$ such that
  \begin{equation}\label{pr:polyRate:ineq}
    \|(P_{\mathbf D_\omega}P_{\mathbf C_\omega})^k(\mathbf x) - P_{\mathbf C_\omega \cap \mathbf D_\omega}(\mathbf x)\|_\omega
    = \|(P_{\mathbf M_2} P_{\mathbf M_1})^k(\mathbf x)\|_\omega
    \leq \frac{C_\omega(\mathbf x)}{\sqrt{k}}.
  \end{equation}
  Note that, by \cite[Theorem 4.6 (5)]{Deutsch2001} (or by \eqref{prop:Minkowski:M} with $r = 2$), we obtain
  \begin{equation}\label{}
    \mathbf C_\omega^\perp \subset \overline{\mathbf C_\omega^\perp + (\mathbf C_\omega \cap \mathbf D_\omega)} = \mathbf M_1^\perp
    \quad \text{and} \quad
    \mathbf D_\omega^\perp \subset \overline{\mathbf D_\omega^\perp + (\mathbf C_\omega \cap \mathbf D_\omega)} = \mathbf M_2^\perp.
  \end{equation}
  Consequently, $\mathbf C_\omega^\perp + \mathbf D_\omega^\perp \subset \mathbf M_1^\perp + \mathbf M_2^\perp$, which, when combined with \eqref{pr:polyRate:ineq}, proves \eqref{pr:polyRate:step1}.

  We may now return to \eqref{th:polyRate:eq}. Let $y \in A_\omega(\mathbf C_\omega^\perp)$ and let $\mathbf y = \{y\}_{i=1}^r$. Then $y = A_\omega(\mathbf x)$ for some $\mathbf x \in \mathbf C_\omega^\perp$ and, by \eqref{lem:ProjCD:D},
  $\mathbf y = P_{\mathbf D_\omega}(\mathbf x) = \mathbf x - P_{\mathbf D_\omega^\perp}(\mathbf x) \in \mathbf C_\omega^\perp + \mathbf D_\omega^\perp$. Consequently, by \eqref{pr:polyRate:step1}, there is $C_\omega(y):=C_\omega(\mathbf y)>0$ such that
  \begin{equation}\label{}
    \|T_\omega^k(y) - P_{M}(y)\|
    = \|(P_{\mathbf D_\omega}P_{\mathbf C_\omega})^k(\mathbf y) - P_{\mathbf C_\omega \cap \mathbf D_\omega}(\mathbf y)\|_\omega
    \leq \frac{C_\omega(y)}{\sqrt{k}}
  \end{equation}
  for all $k=1,2,\ldots$, where the equality follows from \eqref{thm:normConvergence:eq}.
\end{proof}

\section{Appendix} \label{sec:Appendix}
It is well known that if a series $\sum_{i=1}^{\infty} y_i$ in $\mathcal H$ is \emph{absolutely convergent}, that is, when $\sum_{i=1}^{\infty}\|y_i\| <\infty$, then it is also \emph{unconditionally convergent}, that is, $\sum_{i=1}^{\infty} y_{\sigma(i)}$ exists for all bijections $\sigma$ in $\mathbb Z_+$ and equals $\sum_{i=1}^{\infty} y_i$. At this point recall that the unconditionally convergent series coincide with the absolutely convergent series if and only if the space $\mathcal H$ is of finite dimension; see \cite{DvoretzkyRogers1950}.

We slightly strengthen the unconditional convergence in Lemma \ref{lem:rearrangement} below. To this end,  for an absolutely convergent series $\sum_{i=1}^{\infty}y_i$ and for any subset $I = \{i_1, i_2,\ldots\}$ of $\mathbb Z_+$, we formally define $\sum_{i\in I} y_i := \sum_{l=1}^{\# I} y_{i_l}$.  A result similar to Lemma \ref{lem:rearrangement} can be found, for example, in \cite[Theorem 6.3.1]{KrizPultr2013} for $\mathcal H=\mathbb R$.

\begin{lemma}[Rearrangement Lemma]\label{lem:rearrangement}
  Let $q \in \mathbb Z_+ \cup \{\infty\}$ and let $\{I_j\}_{j=1}^q$ consist of nonempty and pairwise disjoint subsets of $\mathbb Z_+$, possibly infinite, such that $\mathbb Z_+ = \bigcup_{j=1}^q I_j$. Assume that the series $\sum_{i=1}^{\infty} y_i$ is absolutely convergent. Then
  \begin{equation}\label{lem:rearrangement:eq}
    \sum_{i=1}^{\infty} y_i
    =  \sum_{j=1}^q \left( \sum_{i \in I_j} y_i \right),
  \end{equation}
  where the summation over $j$, as well as the summations over $i\in I_j$, do not depend on the order of summands.
\end{lemma}

\begin{proof}
  Note that the absolute convergence of the series $\sum_{i=1}^{\infty} y_i$, when combined with the triangle inequality, leads to
    \begin{equation}\label{}
    \sum_{j=1}^q \left\|\sum_{i\in I_j} y_i \right\|
    \leq \sum_{j=1}^{q}\left(\sum_{i\in I_j}\|y_i\|\right)
    = \sum_{i=1}^{\infty}\|y_i\| < \infty,
  \end{equation}
  where the equality holds by \cite[Theorem 6.3.1]{KrizPultr2013}. Consequently, the series $\sum_{i\in I_j} y_i$ converges absolutely, hence unconditionally, to some $z_j \in \mathcal H$, $j=1,\ldots,q$. Furthermore, the series $\sum_{j=1}^{q}z_j$ converges absolutely, hence unconditionally.

  Let now $I=\{i_1,i_2,\ldots\}$, $J=\{j_1,j_2,\ldots\}$ and $K = I \cup J = \{k_1, k_2, \ldots\}$ be countably infinite and increasingly ordered sets of $\mathbb Z_+$ such that $I\cap J = \emptyset$. We claim that
  \begin{equation}\label{pr:rearrangement:2sets}
    \sum_{k\in K} y_k = \sum_{i\in I} y_i + \sum_{j \in J} y_j.
  \end{equation}
  To see this, first define
  \begin{equation}\label{}
    [n]:=\min\{m \colon \{i_1,\ldots,i_n\} \cup \{j_1,\ldots,j_n\} \subset \{k_1,\ldots,k_m\}\}
  \end{equation}
  and
  \begin{equation}\label{}
    M_n := \left\{ \{k_1,\ldots,k_{[n]}\} \setminus
    \big(\{i_1,\ldots,i_n\} \cup \{j_1,\ldots,j_n\}\big)\right \}.
  \end{equation}
  Observe that $\min M_n \geq n$ whenever the set $M_n \neq \emptyset$. Since all the three series in \eqref{pr:rearrangement:2sets} converge, we have
  \begin{equation}
    \left\| \sum_{k\in K} y_k - \sum_{i\in I} y_i - \sum_{j \in J} y_j\right\|
    = \lim_{n\to\infty} \left\| \sum_{l=1}^{[n]} y_{k_l}
    - \sum_{l=1}^{n}y_{i_l} - \sum_{l=1}^{n}y_{j_l}\right\|
    \leq \lim_{n\to\infty} \sum_{i=n}^{\infty} \|y_i\| = 0.
  \end{equation}
  Obviously, formula \eqref{pr:rearrangement:2sets} holds when either one, or both, of $I$ and $J$ are finite.

  By induction, equality \eqref{pr:rearrangement:2sets} carries over to any finite number of sets. In particular, this proves \eqref{lem:rearrangement:eq} for all finite $q \in \mathbb Z_+$. We now show that \eqref{lem:rearrangement:eq} also holds for $q = \infty$. Indeed, redefine
  \begin{equation}\label{}
    [n]:= \min \{m \colon \{1,\ldots,n\} \subset I_1 \cup \ldots \cup I_m\}
  \end{equation}
  and
  \begin{equation}\label{}
    M_n := Z_+ \setminus K_n, \quad \text{where} \quad K_n:=I_1\cup \ldots \cup I_{[n]}.
  \end{equation}
  Note here that since $q=\infty$, we get $M_n \neq \emptyset$ and thus $\min M_n \geq n$. Consequently, by \eqref{pr:rearrangement:2sets} applied to a finite number of sets, first to $K = K_n$ and then to $K = \mathbb Z_+$, we obtain
  \begin{equation}\label{}
    \left\| \sum_{i=1}^{\infty} y_i - \sum_{j=1}^{[n]} z_j\right\|
    = \left\| \sum_{i=1}^{\infty} y_i - \sum_{k\in K_n} y_k\right\|
    = \left\| \sum_{i \in \mathbb Z_+\setminus K_{n}} y_i\right\|
    \leq \sum_{i=n}^{\infty}\|y_i\| \to 0
  \end{equation}
  as $n\to \infty$.
\end{proof}

\textbf{Acknowledgements.}  We are grateful to two anonymous referees for all their comments and remarks which helped us improve our manuscript.  This work was partially supported by the Israel Science Foundation (Grants 389/12 and 820/17), the Fund for the Promotion of Research at the Technion and by the Technion General Research Fund.

\small


\end{document}